\documentclass[10pt]{amsart}
\usepackage{amsfonts}
\usepackage{indentfirst,amsthm,bm,amsmath}
\usepackage{hyperref}
\usepackage{cite}
\usepackage{amssymb}

\newtheorem{theorem}{Theorem}[section]
\newtheorem{corollary}[theorem]{Corollary}
\newtheorem{lemma}[theorem]{Lemma}
\newtheorem{proposition}[theorem]{Proposition}

\newtheorem{remark}[theorem]{Remark}

\begin{document}

\title[]{On the existence of meromorphic solutions of the complex Schr\"{o}dinger equation with a q-shift}

\author[R. Korhonen]{Risto Korhonen}

\address[]{Department of Physics and Mathematics, University of Eastern Finland, P.O.Box~111, FI--80101 JOENSUU, FINLAND}
\email{risto.korhonen@uef.fi}

\author[W. L. Liu]{Wenlong Liu$^*$}
\address[]
{Department of Physics and Mathematics, University of Eastern Finland, P.O.Box~111, FI--80101 JOENSUU, FINLAND}
\email{wenlong.liu@uef.fi}

\thanks{$^*$ Corresponding author.}
\thanks{The second author is supported by the China Scholarship Council(\# 202306820018).}





\begin{abstract}
In this paper, we study the following complex Schr\"{o}dinger equation with a $q$-difference term:
\begin{align}\tag{†}\label{dagger}
f'(z) = a(z)f(qz) + R(z, f(z)), \quad R(z, f(z)) = \frac{P(z, f(z))}{Q(z, f(z))},
\end{align}
where $a(z) \not\equiv 0$ is a small meromorphic function with respect to $f(z)$, and all the coefficient functions of $R(z, f(z))$ are also small meromorphic functions with respect to $f(z)$. 
We assume that $q\in\mathbb{C}\setminus \left \{ 0,-1,1 \right \}  $ and that $R(z, f(z))$ is an irreducible rational function in both $f(z)$ and $z$. We obtain some necessary conditions for \eqref{dagger} to have meromorphic solutions of zero order and non-constant entire solutions, respectively.

In particular, if $R(z,f(z))$ reduces to a polynomial in $f(z)$ with $\deg_f(R)\leq 2$ and all the coefficients are constant, then under this assumption and without imposing any restrictions on the growth order of $f(z),$ we prove the existence of entire solutions in many cases, study their number, and further investigate the local and global meromorphic solutions to \eqref{dagger}. Additionally, we  consider the possible forms of the meromorphic solutions to \eqref{dagger} in certain conditions
and examine exponential polynomials as possible solutions of \eqref{dagger}. 
\end{abstract}

\maketitle

\section{Introduction}

We assume that the reader is familiar with the standard notations of Nevanlinna theory for meromorphic functions (see \cite{laine2011nevanlinna, cherry2001nevanlinna, hayman1964meromorphic}), such as the characteristic function $T(r,f)$, the proximity function $m(r,f)$, and the integrated counting function $N(r,f)$, among others. 
The symbol $S(r,f)$ denotes any quantity satisfying $S(r,f) = o(T(r,f))$ as $r \to \infty$, possibly outside an exceptional set of $r$ of finite linear measure. 
A meromorphic function $a(z)$ is called a \emph{small function} with respect to $f(z)$ if and only if $T(r,a(z)) = S(r,f(z))$.


In recent years, several researchers have applied Nevanlinna theory to investigate the existence and growth of meromorphic solutions to delay differential equations in the complex plane. Representative
works include those of Xu and Cao \cite{xu2021meromorphic}; Cao, together with
Chen and the first author \cite{cao2023meromorphic}; Laine and Latreuch \cite{laine2022remarks}; Chen and Cao \cite{chen2022meromorphic}; Zhang and Huang \cite{zhang2020entire}; Liu and Song \cite{liu2017meromorphic}; Hu and Liu \cite{hu2021malmquist}; Wang, Han, and
Hu \cite{wang2019quantitative}; Cao, together with the present authors \cite{cao2025transcendental}; Halburd and the first author \cite{halburd2017growth}; Li \cite{li2016existence}; Zhang and Huang \cite{zhang2020entire}; Zhao and Huang \cite{zhao2023meromorphic}; Zhang and Liao \cite{zhang2011entire}.
Nevanlinna theory serves as a powerful analytical tool in this context, providing an effective framework for studying both the existence and the growth of such solutions. 
In particular, Halburd and Korhonen \cite{HALBURD2006477}, and independently Chiang and Feng \cite{ChiangFeng2008}, established the difference analogue of the lemma on the logarithmic derivative, which has become a fundamental method in the analysis of meromorphic solutions of differential-difference equations.


It is also noteworthy that Barnett et al. \cite{Barnett2007NevanlinnaTF} established the $q$-difference analogue of the lemma on the logarithmic derivative in 2007. 
This result offers a new perspective for investigating the existence of meromorphic solutions of complex differential equations involving $q$-difference terms. 
Consequently, it is natural to ask whether such equations actually admit meromorphic solutions. The classical paper on the  Schr\"{o}der equation
\begin{align*}
y(qz) = R(y(z)),
\end{align*}
where $q\in \mathbb{C}\setminus \left \{ 0,1 \right \},$ and $R(y)$ is a rational function in $y(z)$, is due to J. Ritt~\cite{Ritt1926}. L. Rubel~\cite{f88b8918-bbff-38b1-8dbf-6a860ea14875} posed the question: What can be said about the more general equation 
$$y(qz)=R(z,y(z)),$$
where $R(z,y(z))$ is rational in both $y(z)$ and $z?$ 
Several studies \cite{valiron1952fonctions, gundersen2002meromorphic, Ishizaki1998, Wittich1949} have examined the existence of meromorphic solutions of the non-autonomous Schr\"{o}der $q$-difference equation
\[
f(qz) = R(z, f(z)),
\]
where $R(z, f(z))$ is rational in both variables.


Agirseven \cite{cf45bf56-3063-3fee-a55f-03d225af2c88} introduced the following Schr\"{o}dinger differential-delay equation:
\[
\begin{cases}
 i\dfrac{dv(t)}{dt} + Av(t) = bAv(t-w) + f(t), \quad t \in (0,\infty),\\[6pt]
 v(t) = \varphi(t), \quad -w \leq t \leq 0,
\end{cases}
\]
where $A$ is a self-adjoint positive definite operator, and $\varphi(t)$ and $f(t)$ are continuous functions. 
From this Schr\"{o}dinger-type differential-difference equation, the following equation,
\begin{equation}\label{New E35}
f'(z) = a(z)f(z+n) + b(z)f(z) + c(z),
\end{equation}
may be regarded as the complex Schr\"{o}dinger equation with delay. 
Similarly, 
\begin{equation}\label{New E36}
f'(z) = a(z)f(qz) + b(z)f(z) + c(z),
\end{equation}
can be viewed as the complex Schr\"{o}dinger equation with a $q$-shift term.


Cao and the two present authors \cite{cao2025transcendental} studied equation \eqref{New E35} and its generalizations, obtaining several necessary conditions for the existence of meromorphic solutions. 
In this paper, we investigate the existence of meromorphic solutions to 
\begin{equation}\label{New E33}
f'(z) = a(z)f(qz) + \frac{P(z, f(z))}{Q(z, f(z))}, \quad R(z, f(z)) = \frac{P(z, f(z))}{Q(z, f(z))},
\end{equation}
and 
\begin{equation}\label{New E34}
f'(z) = a(z)f(qz) + P(z, f(z)), \quad \deg_{f}(P) \leq 2,
\end{equation}
where $a(z) \not\equiv 0$, $q \in \mathbb{C} \setminus \{0,1\}$, and all coefficient functions of $R(z, f(z))$ are small functions of $f(z)$ in the sense of Nevanlinna theory.

The purpose of the present paper is twofold. First, we use Nevanlinna theory to study the necessaty conditions under which \eqref{New E33} admit a non-constant entire solution or a zero order non-constant meromorphic solution, respectively. The second purpose is to investigate the existence of local and global meromorphic solutions of \eqref{New E34} without imposing any restriction on the growth order of the meromorphic solution $f(z),$ which is independent of Nevanlinna theory. In particular, if 
$0<\left | q \right |<1 ,$ we prove the existence of entire solutions in many cases, study their number, and further investigate the local and global meromorphic solutions to \eqref{New E34}. Additionally, we consider exponential polynomials as possible solutions of \eqref{New E34}. Moreover, if $q>1$ and $\deg_f(P)=2,$
we derive the explicit form of the meromorphic solutions of zero order; if $\left | q \right |>1$ and $\deg_f(P)=1,$ \eqref{New E34} has no entire solutions.
Let us recall the definition of an exponential polynomial of the form
\begin{equation}\label{E36}
f(z) = P_{1}(z)e^{Q_{1}(z)} + \dots + P_{k}(z)e^{Q_{k}(z)},
\end{equation}
where $P_{j}(z)$ and $Q_{j}(z)$ are polynomials in $z$. 
Let $t = \max\{\deg(Q_{j}) : Q_{j} \not\equiv 0\}$, and let $\omega_{1}, \dots, \omega_{m}$ be the distinct leading coefficients of the polynomials $Q_{j}(z)$ of maximal degree $t$. 
Then \eqref{E36} can be rewritten as
\begin{equation}\label{E37}
f(z) = H_{0}(z) + H_{1}(z)e^{\omega_{1}z^{t}} + \dots + H_{m}(z)e^{\omega_{m}z^{t}},
\end{equation}
where each $H_{j}(z)$ is either an exponential polynomial of degree $<t$ or an ordinary polynomial in $z$. 
By construction, $H_{j}(z) \not\equiv 0$ for $1 \leq j \leq m$.

The remainder of this paper is organized as follows. Section 2 presents the main results.  Section 3 provides several lemmas that will be used in the proofs of the main results. Sections $4$–$6$ contain the proofs of the three main results.

\section{main results}
In Theorem~\ref{T3}, we apply Nevanlinna theory to investigate the existence of meromorphic solutions of \eqref{New E1} and establish several necessary conditions for their existence. In this theorem, we require that the growth order of $f(z)$ be zero. 
We then obtain that if $R(z,f(z))$ reduces to a polynomial in $f(z)$ and all coefficients of both $R(z,f(z))$ and $a(z)$ are rational functions, then $\deg_{f}(P) \leq 2$. 

\begin{theorem}\label{T3}
Suppose that $f(z)$ is a meromorphic solution of \begin{align}\label{New E1}
{f}'(z)=a(z)f(qz)+\frac{P(z,f(z))}{Q(z,f(z))}, \quad R(z,f(z))=\frac{P(z,f(z))}{Q(z,f(z))}, 
\end{align}
where $R(z,f(z))$ is an irreducible rational function of $f(z)$ with meromorphic coefficients.
Let $a(z)$ and all coefficient functions of $R(z,f(z))$ be small meromorphic functions of $f(z),$ and let $a(z)\not\equiv 0,$ $q\in \mathbb{C}\setminus \left \{ 0,-1,1 \right \}.$ 
\begin{enumerate}
 \item[(i)]Let $0<\left | q \right |<1,$ and let \eqref{New E1} have a non-constant entire solution $f(z)$. 
 \begin{enumerate}
        \item[(a)]
 If $f(z)$ is a factor of $P(z,f(z)),$ then $$\deg_{f}(P)\leq 3,\,\,\,\, \deg_{f}(Q)\leq 2;$$ 
 \end{enumerate}
\begin{enumerate}
        \item[(b)] 
 If $f(z)$ is not a factor of $P(z,f(z)),$ then $$\deg_{f}(P)\leq 2,\,\,\,\, \deg_{f}(Q)\leq 1.$$
\end{enumerate}
\end{enumerate}
\begin{enumerate}
  \item[(ii)] 
Let $f(z)$ be a non-constant meromorphic solution of order $0$ of \eqref{New E1}. 
\begin{enumerate}
        \item[(a)]
If $\deg_{f}(P)>\deg _{f}(Q)+1,$ then $\deg_{f}(Q)\leq 1,\,\,\,\,\deg_{f}(P)\geq 2\deg_{f}(Q).$
\end{enumerate} 
\begin{enumerate}
        \item[(b)]
If $\deg_{f}(P)\leq\deg _{f}(Q)+1,$ then $\deg_{f}(Q)\leq 1,\,\,\,\,\deg_{f}(P)\leq 2.$
\end{enumerate} 
\begin{enumerate}
        \item[(c)]
If $N(r,f(z))=S(r,f(z)),$ then $\deg_{f}(Q)=0, \,\,\,\,\deg_{f}(P)\leq 1.$     
\end{enumerate}
\end{enumerate}

\begin{enumerate}
  \item[(iii)] 
Let $f(z)$ be a transcendental meromorphic solution of \eqref{New E1} of order $0,$ and let $a(z)$ and all the coefficients of $R(z,f(z))$ be rational.
Then we have 
\begin{enumerate}
\item[(a)]  $$\deg_{f}(Q)= 1,\deg_{f}(P)= 3, \,\,\,\text{or}\,\,\, \deg_{f}(Q)= 0,\deg_{f}(P)\leq 2.$$
\end{enumerate}
\begin{enumerate}
\item[(b)]
If $\deg_{f}(P)-\deg_{f}(Q)=2,$ then $\overline{N}(r,f(z))=T(r,f(z))+S(r,f(z)).$
\end{enumerate}
\end{enumerate}
 \end{theorem}
In following Theorems~\ref{T1} and~\ref{T2}, we further examine the existence of meromorphic and entire solutions of \eqref{New E34}.  
Assume now that $R(z,f(z))$ reduces to a polynomial in $f(z)$ whose coefficients are all constant. That is, \eqref{New E1} becomes
\begin{align}\label{E1}
{f}' (z)=Af(qz)+Bf(z)^{2}+Cf(z)+D.     
\end{align}
If $B\ne 0,$ then Theorem \ref{T1} illustrates that, when $0<|q|<1,$ we obtain results on the existence and non-existence of meromorphic (or entire) solutions to \eqref{E1}, whereas for $q>1,$ it gives the precise form of zero-order meromorphic solutions to \eqref{E1}.
If $B=0,$ then Theorem \ref{T2} illustrates that, when $0<|q|<1,$ we obtain results on the existence of entire solutions and the non-existence of meromorphic solutions to \eqref{E1}, whereas for $|q|>1,$ it shows the non-existence of entire solutions to \eqref{E1}.

\begin{theorem}\label{T1}
Let $A,$ $B,$ $C,$  $D,$ and  $q$ be fixed constants, with the conditions that  $A\ne 0,$ $B\ne 0,$ and $q\in \mathbb{C}\setminus \left \{ 0,-1,1 \right \} .$ Then we have following:
\begin{enumerate}
 \item[(i)] Let $0<\left | q \right |<1,$ and let $D=0.$  
 Then \eqref{E1} 
has uncountably many transcendental entire solutions. 

\item[(ii)]
Let $0<\left | q \right |<1,$ and let $D\ne0.$  
\begin{enumerate}
        \item[(a)]
If $f(0)=0,$ then \eqref{E1} has only one transcendental entire solution. 
        \item[(b)] 
If $f(0)\ne 0,$ then \eqref{E1} has uncountably many transcendental entire solutions.        
\end{enumerate}

 \item[(iii)] Let $0<\left | q \right |<1.$ Then  \eqref{E1} has no exponential polynomial solutions.

 \item[(iv)] 
 \begin{enumerate}
        \item[(a)]
 Let $0<\left | q \right |<1.$  Then all poles of $f(z)$ are simple, and  \eqref{E1} admits one local meromorphic solution around $z=0$, which can be extended meromorphically to the entire complex plane $\mathbb{C}.$ In particular, suppose that $C=-\frac{A}{q} .$ If $D\ne 0,$ then \eqref{E1} admits one local transcendental meromorphic solution around $z=0,$ whereas if $D=0,$ then \eqref{E1} admits one local  meromorphic solution of the form $$f(z)=-\frac{1}{Bz} $$ around $z=0.$
 
 \item[(b)]
  For any fixed $z_{0}\ne 0,$ \eqref{E1} has  local meromorphic solution with simple poles
  in a neighborhood of $z=z_{0}$.
\end{enumerate}
 \item[(v)] Let $q>1,$ and let $A,$ $B,$ $C,$ and $D$ be positive real numbers. Assume that \eqref{E1} admits a meromorphic solution of the form: $$f(z)=\frac{t_n}{(z-z_0)^n}+\frac{t_{n-1}}{(z-z_0)^{n-1}}+\cdots +\sum_{j=0}^{\infty}s_j(z-z_0)^j,\,\,\,z_0\in \mathbb{C}, $$
 with $s_0,s_1\in \mathbb{R}.$
 Then all meromorphic solutions of order zero of \eqref{E1} are of the form:
 $$f(z)=-\frac{1}{B}\cdot \frac{1}{z-z_{0}}+s_0.$$
\end{enumerate}
\end{theorem}

\begin{theorem}\label{T2}
Suppose that $A,$ $B,$ 
 $C,$ $D,$ and $q$ are fixed constants such that $A\ne 0, B=0,$ and and $q\in \mathbb{C}\setminus \left \{ 0,1 \right \} .$
 Then we have following:
\begin{enumerate}
\item[(i)] If $0 < \left|q\right| < 1,$ then \eqref{E1} has  uncountably many transcendental entire solutions. However, if $\left|q\right|>1,$ then \eqref{E1} does not have entire solutions. 

\item[(ii)] \eqref{E1} has no meromorphic solution for $0<\left | q \right |<1.$

\end{enumerate}    
\end{theorem}

\section{Preliminary lemmas}

\begin{proposition}[\!\!\cite{bergweiler1998meromorphic}] If  $f$ is meromorphic, then 
\begin{equation}\label{new 1}
T(r,f(qz)) = T(|q|r,f(z)) + O(1)
\end{equation}
for all constants  $q\not=0$.
\end{proposition}

\begin{lemma}\label{L3}
Any non-constant meromorphic function $f(z)$ can omit at most two values in $\mathbb{C}\cup\left \{  \infty \right \}   $ and, if $f(z)$ is non-rational, it takes on every other complex value infinitely many times.
\end{lemma}

The following is the celebrated Logarithmic Derivative Lemma, a fundamental result in Nevanlinna theory. The Second Main Theorem of Nevanlinna is an application of this lemma.
\begin{lemma}\label{New L4}
Let $f(z)$ be a non-constant meromorphic function. Then 
$$m\left ( r,\frac{{f}'(z) }{f(z)}  \right )=O(\log rT(r,f(z)))\quad r\notin E ,$$
where $E$ is of finite linear measure.
\end{lemma}

We now present the $q$-difference analogue of the logarithmic derivative lemma given by Barnett et al. \cite{Barnett2007NevanlinnaTF} for meromorphic functions of order zero.

\begin{lemma}\label{New L2}\cite[Theorem 1.2]{Barnett2007NevanlinnaTF}
Let $f(z)$ be a non-constant meromorphic function of order zero, and let $q\in \mathbb{C}\setminus \left \{ 0,1 \right \}.$ Then $$m\left (r, \frac{f(qz)}{f(z)}  \right )=o(T(r,f(z)))$$
on a set of logarithmic density $1.$
\end{lemma}

\begin{lemma}\cite[Theorem 1.1, 1.3]{zhang2010nevanlinna}\label{New L3}
Let $f(z)$ be a non-constant meromorphic function of order zero, and let $q\in \mathbb{C}\setminus \left \{ 0,1 \right \}.$ Then 
\begin{align*}
N(r,f(qz)) & = (1+o(1))N(r,f(z)),\\\nonumber
T(r,f(qz)) & = (1+o(1))T(r,f(z)),
\end{align*}
on a set of lower logarithmic density one.
\end{lemma}

A $q$-difference analogue of the Clunie lemma is given in \cite[Theorem 2.1]{Barnett2007NevanlinnaTF}; an improvement to this appears in \cite[theorem 2.5]{10.1112/jlms/jdm073} for general $q$-difference polynomials.
We also get a $q$-delay-differential analogue of the Clunie lemma by using a similar
method as in \cite[Theorem 2.5]{10.1112/jlms/jdm073}; see below.
Before presenting a $q$-delay-differential analogue of the Mo'honko theorem\cite[Theorem 10.1.7]{liu2021complex}, we 
recall the definition of a $q$-delay-differential polynomial (see \cite[Page 215]{liu2017meromorphic}). A $q$ delay-differential polynomial in $f(z)$ can be written the following form:
$$P(z,f(z))=\sum_{l\in L}b_{l}(z)f(z)^{l_{0},0}f(c_{1}z)^{l_{1},0}\dots f(c_{v}z)^{l_{v},0}f{'}(z)^{l_{0},1}\dots f^{\mu}(c_{\upsilon }z)^{l_{\upsilon},\mu }$$   
where the coefficients $b_{l}(z)$ are small meromorphic functions with respect to $f(z)$
in the sense that their Nevanlinna characteristic functions are $o(T(r, f))$ on a set
of logarithmic density one.

\begin{lemma}\label{New L5}
Let $f(z)$ be a transcendental meromorphic solution
of order zero of a $q$-delay-differential equation of the form
$$U_q(z, f)P_q(z, f) = Q_q(z, f),$$
where $U_q(z, f),$ $P_q(z, f)$ and $Q_q(z, f)$ are $q$-delay-differential polynomials such that
the total degree deg $U_q(z,f)=n$ in $f(z)$ and its $q$-shifts, whereas $\deg Q_q(z, f)\leq n.$
Moreover, we assume that $U_q(z, f)$ contains just one term of maximal total degree
in $f(z)$ and its $q$-shifts. Then
$$m(r, P_q(z, f))= o(T(r, f))$$
on a set of logarithmic density one.    
\end{lemma}

\begin{lemma}\cite[Theorem 2,2]{Barnett2007NevanlinnaTF}\label{L6}
Let f(z) be a non-constant meromorphic solution of order zero to    
$$P(z,f)=0,$$
where $P(z,f)$ is a $q$-delay -differential polynomial in $f(z).$ If $P(z,a)\not \equiv 0$ for a small function $a(z),$ then 

$$m\left ( r,\frac{1}{f-a}  \right )=S(r,f)$$
on a set of logarithmic density one.
\end{lemma}


The original form of the next lemma appears in \cite[Lemma~2.1]{halburd2017growth}, where it is stated for a delay-differential equation. 
Here, we present its \( q \)-delay differential analogue, following the main idea of the proof given in the original work.

\begin{lemma}\label{New L7}
Let $f(z)$ be a non-rational meromorphic solution of 
\begin{align}\label{E97} 
P(z,f(z))=0
\end{align}
Here, $P(z,f(z))$ denotes a $q$ delay differential polynomial in $f(z)$ with rational coefficients. Let $a_{1},$ \dots, $a_{k}$ be rational functions such that $P(z, a_{j})\not \equiv 0$ for all $j\in \left \{ 1,\dots,k \right \}.$ Assume further that   $q\in \mathbb{C} \setminus \left \{ 0,1 \right \}.$
If there exists $s>0$ and $\tau \in \left ( 0,1 \right ) $ such that \begin{align}\label{E98}
\sum_{j=1}^{k}n\left ( r,\frac{1}{f-a_{j}}  \right )\leq k\tau n(\left | q \right |r,f )+O(1),      
\end{align}
then the order of $f(z)$ is positive.
\end{lemma}

\begin{proof}
We suppose against the conclusion that $\rho (f)=0,$ aiming to obtain a
contradiction. We first show that the assumption $P(z,a_{j})\not \equiv 0$ implies that \begin{align}\label{E71}
m\left ( r,\frac{1}{f(z)-a_{j}(z)}  \right )=S(r,f(z)). \end{align}
This is obvious by using Lemma~\ref{New L4} and Lemma~\ref{New L2} and its proof is similar to the original proof. We will omit its proof here.
To finish the proof, we observe that from the assumption \eqref{E98} it follows that
\begin{align}\label{E99}
\sum_{j  = 1}^{k}N\left ( r,\frac{1}{f-a_{j}}  \right )\leq k(\tau+\epsilon)N\left ( \left | q \right |r,f(z)  \right )+O(\log r)  \end{align}
where $\epsilon>0$ is chosen so that $\tau +\epsilon<1.$ The first main theorem of Nevanlinna
theory now yields
\begin{align}\label{New 70}
kT(r,f(z)) = \sum_{j = 1}^{k}\left (m\left ( r,\frac{1}{f-a_{j}}  \right )+N\left ( r,\frac{1}{f-a_{j}}  \right )   \right )+O(\log r).  
\end{align}
By combining \eqref{E71}, \eqref{E99},  \eqref{New 70}, and \eqref{new 1},  it follows that  
\begin{align}
kT(r,f(z))&\leq k(\tau +\epsilon )N\left ( \left | q \right |r,f(z)  \right )+S(r,f(z))\\\nonumber
&\leq k(\tau +\epsilon )T\left ( \left | q \right |r,f(z)  \right )+S(r,f(z))\\\nonumber
&=k(\tau +\epsilon )T\left ( r,f(qz)  \right )+S(r,f(z))\\\nonumber
&=k(\tau +\epsilon )T\left ( r,f(z)  \right )+S(r,f(z))\\\nonumber
&<kT\left ( r,f(z)  \right )+S(r,f(z)).\\\nonumber
\end{align}
This is a contradiction. Hence, $\rho (f)>0.$
\end{proof}

\begin{corollary}\label{C1}\cite{wen2012exponential}
Let $t$ be a positive integer, $a_{0}(z),$ \dots, $a_{n}(z)$ be either exponential polynomials of degree $<t$ or ordinary polynomials in $z,$ and $b_{1}, \dots, b_{n}\in \mathbb{C}\setminus {0}$ be distinct constants. Then
\begin{align*}
\sum_{j= 1}^{n}a_{j}(z)e^{b_{j}z^{t}}= a_{0}   
\end{align*}
holds only when $a_{0}(z)\equiv \dots \equiv a_{n}(z)\equiv0.$
\end{corollary}

\section{proof of Theorem \ref{T3}}
\begin{proof}

(i) Suppose that $f(z)$ is a non-constant entire solution of \eqref{New E1}. 
Then, \eqref{New E1} can be written as 
\begin{align}\label{E100}
\frac{{f}'(z)}{f(z)} & = a(z)\frac{f(qz)}{f(z)}+\frac{P(z,f(z))}{f(z)\cdot Q(z,f(z))}.   
\end{align}
Let $\frac{P(z,f(z))}{f(z)\cdot Q(z,f(z))}=K(z, f(z)).$ Then \eqref{New E1} becomes 
\begin{align}\label{New E42}
\frac{{f}'(z)}{f(z)} & = a(z)\frac{f(qz)}{f(z)}+K(z, f(z)).   
\end{align}
By using Lemma~\ref{New L4}, we have
\begin{align}\label{E59}
m\left( r, K(z, f(z)) \right) 
  &\leq m\left( r, \frac{f'(z)}{f(z)} \right) + m\left( r, \frac{f(qz)}{f(z)} \right) + S(r, f(z)) \\\nonumber
  &\leq m\left( r, \frac{1}{f(z)} \right) + m(r, f(qz)) + S(r, f(z)).
\end{align}
Additionally, since $f(z)$ is an entire solution of \eqref{New E1},
it follows that
\begin{align}\label{E60}
N\left( r, K(z, f(z)) \right) 
  &\leq N\left( r, \frac{1}{f(z)} \right) + N\left( r, f'(z) - a(z)f(qz) \right) \\\nonumber
  &= N\left( r, \frac{1}{f(z)} \right) + S(r, f(z)).
\end{align}
Combining \eqref{new 1}, \eqref{E59}, \eqref{E60}, the first main theorem of Nevanlinna, and the definition of  small functions, we have
\begin{align}\label{E61} 
T(r,K(z, f(z)))&\leq T\left ( r,f \right)+m(r,f(qz))+S(r,f(z))\\\nonumber
&= T\left ( r,f \right)+T(r,f(qz))+S(r,f(z))\\\nonumber
&= T(r,f)+T\left ( \left | q \right |r,f  \right )+  S(r,f(z)).\\\nonumber
\end{align}
Then we see that 
\begin{align}\label{New E80}
T\left ( r,K(z, f(z)) \right )\leq 2T(r,f(z))+S(r,f).  \end{align}
Noting that $K(z, f(z))=\frac{P(z,f(z))}{f(z)Q(z,f(z))},$ it follows that there are two cases that need to be discussed. Firstly, if $f(z)$ is a factor of $P(z,f(z)),$ then
using \cite[Theorem 2.2.5]{laine2011nevanlinna} yields 
\begin{align}\label{New E81}
T(r,K(z, f(z)))=\max \left \{ \deg_{f}(P)-1,\deg_{f}(Q) \right \}T(r,f(z))+S(r,f(z)). \end{align}
Together with \eqref{New E80} and \eqref{New E81}, we obtain $$\deg_{f}(P)\leq 3,\,\,\,\, \deg_{f}(Q)\leq 2.$$
If $f(z)$ is not a factor of $P(z,f(z)),$ then
using \cite[Theorem 2.2.5]{laine2011nevanlinna} gives 

\begin{align}\label{New E41}
T(r,K(z, f(z)))=\max \left \{ \deg_{f}(P),1+\deg_{f}(Q) \right \}T(r,f(z))+S(r,f(z)).    
\end{align}
Together with \eqref{New E80} and \eqref{New E41} imply that $$\deg_{f}(P)\leq 2,\,\,\,\, \deg_{f}(Q)\leq 1.$$

(ii) We now assume that $f(z)$ is a non-constant meromorphic solution of order $0$ of \eqref{New E1}. From \eqref{E100}, together with Lemma~\ref{New L4} and Lemma~\ref{New L2}, it follows that \begin{align}\label{New 
E30} 
m(r,K(z,f(z)))=S(r,f(z)) 
\end{align} on a set of logarithmic density one. Define two sets $A=\left \{ z:f(z)=\infty  \right \},B= \left \{ z:f(z)\not =\infty  \right \}.$
Then, by applying Lemma~\ref{New L3} to \eqref{E100}, we have
\begin{align}\label{New E26}
N_{B}(r, K(z,f(z))) 
  &= N_{B}\left( r, \frac{f{'}(z) - a(z)f(qz)}{f(z)} \right) \\\nonumber
  &\leq N\left( r, \frac{1}{f(z)} \right) + N_{B}(r,f(qz)) + S(r,f(z)) \\\nonumber
  &\leq N\left( r, \frac{1}{f(z)} \right) + N(r,f(z)) + S(r,f(z))
\end{align}
on a set of lower logarithmic density one. Since $K(f(z))=\frac{P(z,f(z))}{f(z)Q(z,f(z))},$
there are two cases that need to be discussed. 

Firstly, if $f(z)$ is not a factor of $P(z,f(z)),$ then 
\begin{align}\label{New E28}
N_{B}(r, K(z,f(z)))&= N_{B}\left( r, \frac{P(z,f(z))}{f(z)Q(z,f(z))}  \right) \\\nonumber
  &=N_{B}\left ( r,\frac{1}{f(z)Q(z,f(z))}  \right )+S(r,f(z))\\\nonumber
  &=N\left ( r,\frac{1}{f(z)Q(z,f(z))}  \right )+S(r,f(z)).\\\nonumber
 \end{align}
Together, \eqref{New E26} and \eqref{New E28} yield \begin{align}\label{New E37}
N\left ( r,\frac{1}{f(z)Q(z,f(z))}  \right )\leq N\left ( r,f(z) \right )+N\left ( r,\frac{1}{f(z)}  \right )+S(r,f(z)) \end{align}
on a set of lower logarithmic density one.
By \eqref{New E37}, we have \begin{align}\label{New E29}
&N\left ( r, K(z,f(z)) \right )\\\nonumber 
=& N\left ( r,\frac{P(z,f(z))}{f(z)Q(z,f(z))}  \right )\\\nonumber
=&\max \left \{ \deg_{f}(P)-\deg_{f}(Q)-1,0 \right \}N(r,f)+N\left ( r,\frac{1}{f(z)Q(z,f)}  \right )+S(r,f)\\\nonumber
\leq& \max \left \{ \deg_{f}(P)-\deg_{f}(Q)-1,0 \right \}N(r,f)+N(r,f)+N\left ( r,\frac{1}{f}  \right )+S(r,f)    
\end{align}
on a set of lower logarithmic density one.
If $\deg_{f}(P)> \deg_{f}(Q)+1,$ then \begin{align}\label{New E87}
N\left ( r, K(z,f(z)) \right )\leq  (\deg_{f}(P)-\deg_{f}(Q))N(r,f)+N\left ( r,\frac{1}{f}  \right )+S(r,f)   
\end{align}
on a set of lower logarithmic density one.
Together, \eqref{New E30}, \eqref{New E87}, and \cite[Theorem 2.2.5]{laine2011nevanlinna} yield \begin{align}\label{New E31}
T\left ( r, K(z,f(z)) \right ) & = T\left ( r,\frac{P(z,f(z))}{f(z)Q(z,f(z))}  \right )\\\nonumber
 &=\max \left \{ \deg_{f}(P),1+\deg_{f}(Q)  \right \}T(r,f(z))+S(r,f) \\\nonumber
&\leq (\deg _{f}(P)-\deg _{f}(Q))N(r,f)+N\left ( r,\frac{1}{f}  \right )+S(r,f) \\\nonumber
&\leq (\deg _{f}(P)-\deg _{f}(Q)+1)T(r,f)+S(r,f) 
\end{align} 
on a set of logarithmic density one or of infinity linear measure. This means 
\begin{align}\label{New E89} 
\deg_{f}(Q)\leq 1,\,\,\,\,\deg_{f}(P)\geq 2\deg _{f}(Q).
\end{align}
If $\deg_{f}(P)\leq  \deg_{f}(Q)+1,$ then by \eqref{New E29} we have 
\begin{align}\label{New E88}
 N(r,K(f(z)))\leq N(r,f)+N\left ( r,\frac{1}{f}  \right )+S(r,f)    \end{align}
on a set of lower logarithmic density one. 
Together, \eqref{New E30}, \eqref{New E88}, and \cite[Theorem 2.2.5]{laine2011nevanlinna} yield \begin{align*}
T\left ( r, K(z,f(z)) \right ) & = T\left ( r,\frac{P(z,f(z))}{f(z)Q(z,f(z))}  \right )\\\nonumber
 &=\max \left \{ \deg_{f}(P),1+\deg_{f}(Q)  \right \}T(r,f(z)) \\\nonumber
&\leq N(r,f)+N\left ( r,\frac{1}{f}  \right )+S(r,f) \\\nonumber
&\leq 2T(r,f)+S(r,f) 
\end{align*} 
on a set of lower logarithmic density one or of infinity linear measure.
This means 
\begin{align}\label{New E90} 
\deg_{f}(P)\leq 2,\,\,\,\,\deg_{f}(Q)\leq 1.  
\end{align}
Additionally, if $N(r,f(z))=S(r,f(z)),$ then by \eqref{New E29} we see that 
\begin{align}\label{New E38}
N(r,K(f(z)))\leq N\left ( r,\frac{1}{f(z)}  \right )+S(r,f(z)). 
\end{align}
Together, \eqref{New E30}, \eqref{New E38}, and \cite[Theorem 2.2.5]{laine2011nevanlinna} yield 
\begin{align}\label{New E39}
T(r,K(f(z))) & = \max \left \{ \deg_{f}(P),\deg_{f}(Q)+1   \right \}T(r,f(z))\\\nonumber
&\leq N\left ( r,\frac{1}{f(z)}  \right )+S(r,f(z))\\\nonumber
&\leq T(r,f(z))+S(r,f(z))
\end{align}
on a set of lower logarithmic density one or of infinity linear measure.
This means 
$$\deg_{f}(P)\leq 1,\,\,\,\, \deg_{f}(Q)=0.$$

We now assume that $f(z)$ is a factor of $P(z,f(z)).$ Then \begin{align}\label{New E86} 
K(z,f(z))=\frac{P(z,f(z))}{f(z)Q(z,f(z))}=\frac{\hat P(z,f(z))}{Q(z,f(z))}, \end{align} 
where $\hat P(z,f(z))=\frac{P(z,f(z))}{f(z)}$ is a polynomial in $f(z).$
By the definition of set $B,$ it follows that
\begin{align}\label{New E82}
 N_{B}\left ( r,R(z,f(z)) \right )=N_{B}\left ( r,\frac{P(z,f(z))}{Q(z,f(z))}  \right )=N\left ( r,\frac{1}{Q(z,f(z))}  \right )+S(r,f(z)).     
\end{align}
Also, by applying Lemma~\ref{New L3} and \eqref{New E1}, it follows that

\begin{align}\label{New E83}
 N_{B}\left ( r,R(z,f(z)) \right )=&N_{B}\left ( r,{f}'(z)-a(z)f(qz)  \right )\\\nonumber
=&N_{B}(r,f(qz))+S(r,f(z))\\\nonumber
\leq& N(r,f(z))+S(r,f(z))
\end{align}
on a set of lower logarithmic density one. Together, \eqref{New E82} and \eqref{New E83} yield 
\begin{align}\label{New E84}
N\left ( r,\frac{1}{Q(z,f(z))}  \right )\leq N(r,f(z))+S(r,f(z))\end{align}
on a set of lower logarithmic density one.
In this case, by combining \eqref{New E86} with \eqref{New E84}, we have \begin{align}\label{New E85}
&N\left ( r,K(z,f(z))  \right )\\\nonumber
=&\max \left \{ \deg _{f}(P)-\deg _{f}(Q)-1,0 \right \}N(r,f(z))+N\left ( r,\frac{1}{Q(z,f(z))}  \right )+S(r,f(z))\\\nonumber  
\leq &\max \left \{ \deg _{f}(P)-\deg _{f}(Q)-1,0 \right \}N(r,f(z))+N(r,f)+S(r,f)\\\nonumber
\end{align}
on a set of lower logarithmic density one.
If $\deg_{f}(P)>\deg_{f}(Q)+1,$ then by combining \eqref{New E30}, \eqref{New E85}, with \cite[Theorem 2.2.5]{laine2011nevanlinna}, 
we have \begin{align}\label{New E91}
&T\left ( r,K(z,f(z))  \right )\\\nonumber
=&\max \left \{ \deg_{f}(P)-1,\deg_{f}(Q)   \right \}T(r,f(z))+S(r,f) \\\nonumber 
\leq& (\deg_{f}(P)-\deg_{f}(Q))N(r,f(z))+S(r,f(z))\\\nonumber
\leq& (\deg_{f}(P)-\deg_{f}(Q))T(r,f(z))+S(r,f(z))\\\nonumber
\end{align}
on a set of lower logarithmic density one or of infinity linear measure.
This means \begin{align}\label{New E93} 
\deg_{f}(P)\geq 2\deg_{f}(Q),\,\,\,\,\deg_{f}(Q)\leq 1.
\end{align}
If $\deg_{f}(P)\leq\deg_{f}(Q)+1,$ then by combining \eqref{New E30}, \eqref{New E85}, with \cite[Theorem 2.2.5]{laine2011nevanlinna}, 
we have \begin{align}\label{New E94}
&T\left ( r,K(z,f(z))  \right )\\\nonumber
=& \max \left \{ \deg_{f}(P)-1,\deg_{f}(Q)   \right \}T(r,f(z))+S(r,f(z))\\\nonumber
\leq& N(r,f)+S(r,f).\\\nonumber
\leq& T(r,f)+S(r,f)\\\nonumber
\end{align}
on a set of lower logarithmic density one or of infinity linear measure.
This means 
\begin{align}\label{New E95}
\deg_{f}(P)\leq 2,\,\,\,\,\deg _{f}(Q)\leq 1.
\end{align}
Further, if $N(r,f(z))=S(r,f(z)),$ then by \eqref{New E85} we see that 
\begin{align}\label{New E40}
N(r,K(f(z)))=S(r,f(z)). 
\end{align}
Together, \eqref{New E30}, \eqref{New E40}, and \cite[Theorem 2.2.5]{laine2011nevanlinna} yield 
\begin{align*}
T(r,K(f(z))) & = \max \left \{ \deg_{f}(P)-1,\deg_{f}(Q)   \right \}T(r,f(z))=S(r,f(z))\\\nonumber
\end{align*}
on a set of lower logarithmic density one or of infinity linear measure.
This means 
$$\deg_{f}(P)\leq 1,\,\,\,\, \deg_{f}(Q)=0.$$

From the above analysis, we see that, regardless of whether $f(z)$ is a factor of $P(z,f(z)),$ the following holds:
If 
$\deg_{f}(P)>\deg_{f}(Q)+1,$ then
$$\deg_{f}(Q)\leq 1,\,\,\,\,\deg_{f}(P)\geq 2\deg_{f}(Q).$$ 
If $\deg_{f}(P)\leq \deg_{f}(Q)+1,$ then
$$\deg_{f}(Q)\leq 1,\,\,\,\,\deg_{f}(P)\leq 2.$$ 
If $N(r,f(z))=S(r,f(z)),$ then 
$$\deg_{f}(Q)=0,\,\,\,\, \deg_{f}(P)\leq 1.  $$

(iii) 
(a) We first prove that if $0<\left | q \right |<1 ,$ then 
$\deg_{f}(Q)= 1,\deg_{f}(P)= 3, \,\,\,\text{or}\,\,\, \deg_{f}(Q)= 0,\deg_{f}(P)\leq 2.$ 

Applying Lemma~\ref{New L3} and \cite[Theorem 2.2.5]{laine2011nevanlinna} to \eqref{New E1}, it follows that
\begin{align}\label{New E44}
\max\left\{ \deg_f(P), \deg_f(Q) \right\} T(r, f)
&= T(r, R(z, f(z))) \nonumber \\
&= T(r, f'(z) - a(z) f(qz)) \nonumber \\
&\leq 2T(r, f(z)) + (1 + o(1)) T(r, f(z)) + S(r, f(z)) \nonumber \\
&= 3T(r, f(z)) + S(r, f(z)) \nonumber \\
\end{align}
on a set of lower logarithmic density one. 
By (ii) we see that if the order of $f(z)$ is zero, then $\deg_{f}(Q)\leq 1.$ 
We now prove that if $\deg_{f}(Q)= 1,$ then $\deg_{f}(P)=\deg_{f}(Q)+2=3.$
Let $\deg_{f}(Q)= 1.$ Then equation \eqref{New E1} can be written in the following form:
\begin{align}\label{E80}
{f}'(z) & = a(z)f(qz)+\frac{P(z,f(z))}{f(z)-b(z)},  
\end{align}
where 
$b(z)$ is a rational function.
Applying Lemma~\ref{L6} to \eqref{E80} yields 
\begin{align}\label{E81}
N\left ( r,\frac{1}{f(z)-b(z)}  \right ) & = T(r,f(z))+S(r,f(z))
\end{align}
on a set of logarithmic density one. Note that $a(z)$ and all the coefficients of $R(z,f(z))$ are rational. This implies that all the zeros and poles of $a(z),$ all the zeros and poles of the coefficients of $R(z,f(z)),$ as well as those arising from their finite iterations, can be contained within a large bounded disk $D.$
Then, by \eqref{E81}, we can choose a sufficiently large $\left | z_{0} \right | $ such that $f(z_{0})-b(z_{0})$ has a zero of multiplicity $t\ge 1$, where $z=z_{0}$ is neither a zero nor a pole of $a(z),$ nor of any coefficient of $R(z,f(z)),$ nor of any finite iteration of $a(z)$ or the coefficients of $R(z,f(z)).$
Such a point $z=z_{0}$ is called a \textit{generic zero} of $f(z)-b(z),$ and it satisfies that $z=q^{n}z_{0}$ lies outside the disk $D,$ where $0<\left | q \right |<1 $ and $n$ is an arbitrarily given positive integer.
Then we can iterate \eqref{E80} $n$ times and the zeros and poles of $a(z),$ the coefficients of $R(z,f(z))$ and their iterated terms do not influence the number of poles of $f(q^{n}z)$ in the iteration of \eqref{E80}. Then, by \eqref{E80}, we have $f(qz_{0})=\infty^{t}.$ By iterating \eqref{E80} one step, we have \begin{align}\label{E82}
{f}'(qz) & = a(qz)f(q^{2}z)+\frac{P(qz,f(qz))}{f(qz)-b(qz)}.  
\end{align}
We assume now that 
$\deg_{f}(P)\leq \deg_{f}(Q)=1.$ Then $f(q^{2}z_{0})=\infty^{t+1}.$ By continuing to iterate \eqref{E82}, we have \begin{align}\label{E83}
{f}'(q^{2}z) & = a(q^{2}z)f(q^{3}z)+\frac{P(q^{2}z,f(q^{2}z))}{f(q^{2}z)-b(q^{2}z)}.  
\end{align}
Then $f(q^{3}z_{0})=\infty^{t+2}.$ Note that $z=q^{n}z_{0}$ still lies outside of the disk $D.$ Hence, we can continue the above iteration without considering the influence of the zeros or poles of  $a(z),$ the coefficients of $R(z,f(z)),$ or their finite iteration. It then follows that $f(q^{n}z_{0})=\infty^{t+n-1},$ where $n$ is an arbitrarily given positive integer.

Note that $0<\left | q \right |<1 .$ For a given $n,$  after iterating \eqref{E80} for multiple steps, the coefficients maybe cancel part of the poles of $f(q^{m}z),$ where $m\geq n+1.$ We define
$$ m_{j}:=\underset{i=1,\dots,s}{\max }\left \{ l_{i}\in\mathbb{N}:k_{i}(z_{0})=0^{l_{i}}  \right \}  ,$$
where $k_{i}(z)$ denotes $a(z),$ the coefficients of $R(z,f(z)),$ and their iterated terms. Since $a(z)$ and the coefficients of $R(z,f(z))$ are rational, then we can always find such a sufficiently large $\left | z_{0} \right | $ such that $m_{j}<t+n-1,$ i.e.,
the zeros of the coefficients cannot completely cancel the poles of $f(q^{n}z_{0}).$ 
In other words, we can see that $z=z_{0}$ is still a pole of $f(q^{n}z).$
Hence, we can always find poles of $f(z)$ in any neighborhood of the origin. This implies that $f(z)$ cannot be a meromorphic function, which contradicts the assumptions of Theorem \ref{T3}.
Hence, $\deg_{f}(P)>\deg_{f}(Q)=1.$
In the case where $\deg_{f}(P)=\deg_{f}(Q)+1=2,$ we can proceed in the same way and conclude that $f(z)$ is not a meromorphic function. Additionally, from \eqref{New E44} we see that $\deg_{f}(P)\leq3.$   
Consequently, if $\deg_{f}(Q)= 1,$ then $$\deg_{f}(P)=\deg_{f}(Q)+2= 3.$$

We now prove that if $\deg_{f}(Q)=0,$ then $\deg_{f}(P)\leq 2.$ Let $\deg_{f}(Q)=0.$
There are two cases that need to be discussed.
If $N(r,f)=S(r,f),$
then by (ii), we see that $\deg_{f}(P)\leq 1.$ Assume now that $N(r,f)\ne S(r,f).$ 
Therefore, there exists a point $\hat z_{0}$ such that $f(\hat z_{0})=\infty^{l}$ ($\geq 1$), but $z=\hat z_{0}$ is neither a zero nor a pole of $a(z),$ nor of any coefficient of $R(z,f(z)),$ nor of any finite iteration of $a(z)$ or the coefficients of $R(z,f(z)).$ That is, $z=\hat z_{0}$ is a generic pole of $f(z).$ 
In this case, \eqref{New E1} reduces to 
\begin{align}\label{E84} 
{f}'(z) & = a(z)f(qz)+P(z,f(z)).\end{align}
Suppose that $\deg_{f}(P)=3.$
Then $f(q\hat z_{0})=\infty^{3l}.$ 
Iterating \eqref{E84} by one step yields, 
\begin{align}\label{E85} 
{f}'(qz) & = a(qz)f(q^{2}z)+P(qz,f(qz)).\end{align}
Then $f(q^{2}\hat z_{0})=\infty^{9l}.$
Continuing to iterate \eqref{E85} yields
\begin{align}\label{E86} 
{f}'(q^{2}z) & = a(q^{2}z)f(q^{3}z)+P(q^{2}z,f(q^{2}z)).\end{align}
Then $f(q^{3}\hat z_{0})=\infty^{27l}.$ By continuing to iterate \eqref{E86}, from the above analysis we see that $f(z)$ is not a meromorphic function. This is a contradiction. That is, $\deg_{f}(P)\leq 2.$
Therefore, we can always conclude that if $\deg_{f}(Q)=0,$ then $\deg_{f}(P)\leq 2.$

We now prove that if $\left | q \right |>1 ,$ then 
$\deg_{f}(Q)= 1,\deg_{f}(P)= 3, \,\,\,\text{or}\,\,\, \deg_{f}(Q)= 0,\deg_{f}(P)\leq 2.$

By (ii) we see that if $f(z)$ is order zero, then
$\deg_{f}(Q)\leq 1.$ 
Let $\deg_{f}(Q)=0.$
Then \eqref{New E1} becomes 
\begin{align}\label{New E70} 
{f}'(z)=a(z)f(qz)+P(z,f(z)). 
\end{align}
If $N(r,f(z))=S(r,f(z)),$ then by (ii) we find that $\deg_{f}(P)\leq 1.$ We now consider the case in which $N(r,f(z))\ne S(r,f(z)).$ By \eqref{New E44} we see that $\deg_{f}(P)\leq 3.$ 
Let $\deg_{f}(P)=3,$ and let $f(z)$ have a generic pole of order $k$ at $z=\widetilde{z}.$
Then $f' (\widetilde{z} )=\infty^{k+1}$ and $f (q\widetilde{z} )=\infty^{3k},$ and the starting point of iterative sequence $z=\widetilde{z} $ satisfy that $\left | \widetilde{z} \right | $ is large enough
and lies outside the closed disk $\overline D$. 
Note that  $\left | q \right |>1.$ Therefore, the zeros and poles of $a(z),$ the coefficients of $P(z,f(z)),$ and their $n$ times iterated terms do not influence the number of poles of $f(q^{n}z)$ ($n\geq 1$) in the iteration of \eqref{New E70}.
By iterating \eqref{New E70} one step, it follows that \begin{align}\label{New E71} 
{f}'(qz)=a(qz)f(q^{2}z)+P(qz,f(qz)). 
\end{align}
Then $f' (q\widetilde{z} )=\infty^{3k+1}$ and $f (q^{2}\widetilde{z} )=\infty^{9k}.$
By iterating \eqref{New E70} more steps, we have
$f(q^{n}z)=\infty ^{3^{n}k}.$ 
Therefore, 
\begin{align}\label{E87}
\rho (f)\geq \lambda \left ( \frac{1}{f}  \right )&=\limsup_{r \to \infty}\frac{\log n(r,f)}{\log r}\\\nonumber
&\geq \limsup_{n \to \infty}\frac{\log  (3^{n}k)} {\log \left | q^{n}\hat z \right | }=\frac{\log 3 }{\log \left | q \right |} >0.
\end{align}
This is a contradiction with our assumption. Hence, $\deg_{f}(P)\leq 2.$
Let $\deg_{f}(Q)=1.$ Then \eqref{New E1} becomes
\begin{align}\label{E75}
{f}'(z) & = a(z)f(qz)+\frac{P(z,f(z))}{f(z)-b(z)},  
\end{align}
where $b(z)$ is a rational function.
If $\deg_{f}(P)\leq 1,$ then by Lemma~\ref{L6},  there exists a generic zero of $f(z)-b(z)$
of order $k$ at $z=z_{1}.$ On the other hand,
note that $a(z)$ and the coefficients of $R(z,f(z))$ are rational, then
for a non-generic zero $z=\hat z_{0}$ of $f(z)-b(z),$  we see that $z=\hat z_{0}$ must be a interior point of a large bounded disc $D(0,R)$ and  
\begin{align}\label{E78}
n\left ( R,\frac{1}{f(z)-b(z)}   \right ) =O(1).    
\end{align}
Then $f(qz_{1})=\infty^{k}.$ Iterating \eqref{E75} by one step yields,
\begin{align}\label{E76}
{f}'(qz) = a(qz)f(q^{2}z)+\frac{P(qz,f(qz))}{f(qz)-b(qz)},  
\end{align}
Then $f(q^{2}z_{1})=\infty^{k+1}.$
Continuing to iterate \eqref{E76} yields
\begin{align}\label{E77}
{f}'(q^{2}z) = a(q^{2}z)f(q^{3}z)+\frac{P(q^{2}z,f(q^{2}z))}{f(q^{2}z)-b(q^{2}z)}.  
\end{align}
Then $f (q^{3}z_{1})=\infty^{k+2}.$ Then 
\begin{align}\label{E79}
 \frac{n\left ( r,\frac{1}{f(z)-b(z)}  \right ) }{n(\left | q \right |^{2}r,f(z))}=\frac{n\left ( r,\frac{1}{f(z)-b(z)}  \right ) }{n(r,f(q^{2} z))}\leq \frac{1}{2}+\epsilon   
\end{align}
for all $r\geq r_0,$ where $r_0$ is large enough, and $\epsilon$ is a arbitrarily positive constant.
Together with \eqref{E78} and \eqref{E79} imply that for every zero of $f(z)-b(z),$ we have the following:\begin{align*}
n\left ( r,\frac{1}{f(z)-b(z)}  \right )\leq\frac{1}{2}n(\left | q \right |^2 r,f(z))+O(1).      
\end{align*}
By Lemma~\ref{New L7}, we obtain that $\rho (f)>0.$
This is a contradiction with our assumption. Hence, $\deg_{f}(P)\geq 2.$ If $\deg_{f}(P)=2,$ then we can proceed with the same way to obtain a contradiction. Together, \eqref{New E44} implies that $\deg_{f}(P)=3.$

(b) From \eqref{New E44} we see that there are two cases that need to be discussed. 

We first consider the case where $\deg_{f}(P)=3,$ $\deg_{f}(Q)=1.$ In this case, \eqref{New E1} becomes 
\begin{align}\label{E88}
{f}'(z) = a(z)f(qz)+\frac{b_{0}(z)f(z)^{3}+b_{1}(z)f(z)^{2}+b_{2}(z)f(z)+b_{3}(z)}{f(z)-b_{4}(z)}, 
\end{align}
where $b_{0}(z)\not\equiv 0,$  $b_{3}(z)$ and $b_{4}(z)$ are not simultaneously identically zero. Equation \eqref{E88} can be written as follows:
\begin{align}\label{E89}
f(z)\left [ b_{0}(z)f(z)^{2}+b_{1}(z)f(z)+b_{2}(z) \right ] & = ({f}'(z)-a(z)f(qz) )(f(z)-b_{4}(z))-b_{3}(z). 
\end{align}
Applying Lemma~\ref{New L5} to \eqref{E89} yields $m(r,f(z))=S(r,f(z))$ on a set of logarithmic density one. This means \begin{align}\label{New E92}
N(r,f(z))=T(r,f(z))+S(r,f(z))
\end{align}
on a set of logarithmic density one.

If $\deg_{f}(P)=2,$ $\deg_{f}(Q)=0,$ then \eqref{New E1} becomes \begin{align}\label{E91}
{f}'(z)=a(z)f(qz)+c_{0}(z)f(z)^{2}+c_{1}(z)f(z)+c_{2}(z),   \end{align}
where $c_{0}(z)\not\equiv0.$ Equation \eqref{E91} can be written as follows:
\begin{align}\label{E92} 
f(z)(c_{0}(z)f(z)+c_{1}(z)) & = {f}'(z)-a(z)f(qz)-c_{2}(z). 
\end{align}
Applying Lemma~\ref{New L5} to \eqref{E92} yields $m(r,f(z))=S(r,f(z))$ on a set of logarithmic
density one. We deduce that
\begin{align}\label{E94} 
N(r,f(z))=T(r,f(z))+S(r,f(z))
\end{align}
on a set of logarithmic density one.

We have proved that $N(r, f(z)) =T(r, f) + S(r, f(z))$ holds on a set of logarithmic density one whenever $\deg_f(P) - \deg_f(Q) = 2.$ Additionally, since $a(z)$ and all the coefficients of $R(z, f(z))$ are rational, there exists a generic pole of $f(z)$ of order
$t$ such that $f(\hat z)=\infty^{t}$ and $z=\hat z$ is neither a zero nor a pole of $a(z)$ nor of any
coefficient of $R(z, f(z)).$ Then ${f}'(\hat z)=\infty^{t+1} $ and $f(q\hat z) = \infty^{2t}
.$ We now assume
that $t \geq 2.$ By iterating \eqref{New E1} one step, it follows that
\begin{align}\label{E95}
{f}'(qz) & = a(qz) f(q^{2}z)+\frac{P(qz,f(qz))}{Q(qz,f(qz))} 
\end{align}
Then ${f}'(q\hat z)=\infty^{2t+1}$ and $f(q^{2}\hat z)=\infty^{4t}.$ By iterating \eqref{New E1} for further steps, we can end up with $f(q^{n}\hat z)=\infty^{2^{n}t}.$ From the above analysis, we can see that if $\left | q \right |>1,$ then we can obtain that the growth order of $f(z)$ is positive; whereas if $0<\left | q \right |<1,$ then $f(z)$ is not a meromorphic function. Therefore, we can always obtain a contradiction. Then $t=1.$ That is,
\begin{align}\label{E96}
 \overline{N}(r,f(z))=T(r,f(z))+S(r,f(z)).    
\end{align}

\end{proof}

\section{proof of Theorem \ref{T1}}
Firstly, we present the following propositions that will be used to the proof.
\begin{proposition}\label{P1}\cite{gundersen2002meromorphic}
If $\left |q\right|>1,$ then any local meromorphic solution around the origin of   \begin{align}\label{E11}
f(qz)=\frac{\sum_{j=0}^{p}a_{j}(z)f(z)^{j}}{\sum_{j=0}^{q}b_{j}(z)f(z)^{j}},     
\end{align} where $a_{j}(z)$ and $b_{j}(z)$ are meromorphic functions, 
has a meromorphic continuation over the whole complex plane.  

\end{proposition}
Following the proof of  Proposition~\ref{P1} in \cite{gundersen2002meromorphic}, we find Proposition~\ref{P1} still holds if we replace \eqref{E11} with \eqref{E1}. We give the details for the convenience of the reader.
Moreover, in the case $0<\left|q\right|<1,$ we can replace $z$ with $\frac{1}{q}z$ to obtain the following:
\begin{proposition}\label{P2}
If $\left |q\right|\ne 0, 1,$ then any local meromorphic solution of \eqref{E1} around the origin has a meromorphic continuation over the whole complex plane.  
\end{proposition}
\begin{proof}
Suppose that $\left | q \right |>1   .$
Let $f(z)$ be a meromorphic solution of \eqref{E1} in $\left | z \right |<R $ for some $0<R<\infty .$ From \eqref{E1}, $f(z)$ extends meromorphically to the larger disc $\left | z \right |<\left | q \right |R  .$  Inductively, it follows that $f(z)$ can be continued meromorphically to $$\bigcup_{j=0}^{\infty}\left \{ \left | z \right |<\left | q \right |^{j}R   \right \}=\mathbb{C}.$$ 
Suppose now that $0<\left | q \right |<1.$ Then substituting $\frac{z}{q} $ for $z$ in \eqref{E1} gives \begin{align}\label{E101}
{f}'\left ( \frac{z}{q}  \right )=Af(z)+Bf\left ( \frac{z}{q}  \right )^2 +Cf\left ( \frac{z}{q}  \right )+D.       
\end{align}
Since $f(z)$ is meromorphic in $\left | z \right |<R$ with $0<R<\infty ,$ it follows from \eqref{E101} that  $f(z)$ can be extended meromorphically to the larger disc $\left | z \right |<\frac{1}{\left | q \right | }R  .$ Inductively, $f(z)$ can be extended meromorphically to 
$$\bigcup_{j=0}^{\infty}\left \{ \left | z \right |<\left | \frac{1}{q}  \right |^{j}R   \right \}=\mathbb{C}.$$ 
\end{proof}

\begin{proof}[Proof of Theorem \ref{T1}]
 Let $f(z)=\frac{g(z)}{B}.$ Then \eqref{E1} becomes
\begin{align}\label{E19}
{g}'(z)=Ag(qz)+g(z)^{2}+Cg(z)+BD.      
\end{align}
In order to prove the existence of entire solutions and meromorphic solutions of \eqref{E1}, 
we only need to prove that \eqref{E19} has  entire solutions and meromorphic solutions.
To this end, consider first a formal power series 
\begin{align}\label{E2}
g(z)=\sum_{n=0}^{\infty}a_{n}z^{n}.    
\end{align}
Then \begin{align}\label{E17}
{g}'(z)= \sum_{n= 1}^{\infty}na_{n}z^{n-1}, g(qz)=\sum_{n=0}^{\infty}a_{n}q^{n} z^{n},g(z)^{2}= \sum_{n=0}^{\infty}\left(\sum_{i,j=0}^{i+j=n}a_{i}a_{j}\right)z^{n}.       
\end{align}
By substituting \eqref{E2} and \eqref{E17} into \eqref{E19}, it follows that \begin{align}\label{E18}
&a_{1}= a_{0}^{2}+(A+C)a_{0}+BD,  \\\nonumber  
&2a_{2}=Aa_{1}q+a_{0}a_{1}+a_{1}a_{0}+Ca_{1},\\\nonumber
&3a_{3}=Aa_{2}q^{2}+a_{0}a_{2}+a_{1}^{2}+a_{2}a_{0}+Ca_{2},\\\nonumber
&\dots\\\nonumber 
&(n+1)a_{n+1}=Aa_{n}q^{n}+\sum_{i,j=0}^{i+j=n}a_{i}a_{j}+Ca_{n},\,\,\,n\geq 1.    
\end{align}
 We next consider the two cases where $D=0$ and where $D\ne 0.$

(i). $\mathbf{Case}$ $\mathbf{1}:$ Suppose that $D=0.$ Combining \eqref{E18} with the fact that $A,$ $C,$ and $q$ are fixed constants, we have $a_{n}$ is a polynomial in $a_{0}$ for all $n\ge 0.$ Note that for arbitrary $n\in \mathbb{N},$ $a_n$
forms a one-parameter family determined by $a_0.$ We will prove that there exist uncountably many $a_0$
such that \eqref{E2} converges to a transcendental entire function. That is, \eqref{E19} has uncountably many
transcendental entire solutions.

$\mathbf{Step}$ $\mathbf{1}:$
Let $g(z)\not\equiv0.$
We first prove that there exist uncountably many $a_{0}$ such that each such $a_0$  satisfies that there are infinitely many indices $n$ for which $a_{n}=P_{n}(a_{0})\ne 0.$ 
Let $$A_n=\left \{ a_0\in\mathbb{C}:a_n=P_n(a_0)=0  \right \} ,$$ and let $$A=\bigcup_{n=0}^{\infty}A_n .$$
Define $M=\mathbb{C}\setminus A .$ Note that each $A_n$ is a finite set, and $A$ is the union of countably many finite sets, we see that $A$ is an countable set. Then $M$ is an uncountably set. By the definition of set $M,$ we have $a_n=P_n(a_0)\ne 0$
for every $a_0\in M$ and each $n\in \mathbb{N}\cup \left \{ 0 \right \} .$

$\mathbf{Step}$ $\mathbf{2}:$
Next, we prove that \eqref{E2} is a convergent power series, with $a_{0}\in \mathbb{C}$ fixed so that $g(z)$ has infinitely many nonzero coefficients $a_{n},$ where the existence of such an $a_{0}$ was established in Step $1.$ 

$\mathbf{Step}$ $\mathbf{2.1}:$
Let $a_{0}=\beta.$  By \eqref{E18}, we conclude that $\beta=a_{0}\ne0;$ Otherwise, $a_{n}\equiv 0$ for all $n\in \mathbb{N}\cup\left \{ 0 \right \},$ which is a contradiction with Step $1.$
 
We now consider two cases for $\left |\beta  \right |$ where $0<\left |\beta  \right |=\left | a_{0} \right | \leq 1$ and $\left |\beta  \right |=\left | a_{0} \right | >1.$

$\mathbf{Step}$ $\mathbf{2.2}:$
We will now prove an inequality for $|a_n|$ using induction.
If $0<\left | a_{0} \right |= \left | \beta  \right |\leq 1,$ then we suppose that 
\begin{align}\label{E20}
\left | a_{n} \right |\leq (\left | A \right |+\left | C \right |+1  )^{n},\,\,\,\forall n\leq k, \,\,\,k\in \mathbb{N}\cup \left \{ 0 \right \}.   
\end{align}
It is easily seen that \eqref{E20} holds for $n=0.$ Hence, we can assume $k\geq 1.$ We now prove that \eqref{E20} still holds when $n=k+1.$
By \eqref{E18} we have \begin{align}\label{E51}
a_{k+1}=\frac{1}{k+1}\left [ Aa_{k}q^{k}+\sum_{i,j=0}^{i+j=k}a_{i}a_{j}+Ca_{k}   \right ].    
\end{align}
Combining \eqref{E20}, \eqref{E51}, and $0<\left | q \right |<1 ,$ we have
\begin{align}\label{E21}
&\left | a_{k+1} \right |\\\nonumber
\leq& \frac{1}{k+1}\left [\left | A \right | \left | q \right |^{k}(\left | A \right |+\left | C \right |+1)^{k}+(k+1)(\left | A \right |+\left | C \right |+1)^{k}+\left | C \right | (\left | A \right |+\left | C \right |+1)^{k}\right ]\\\nonumber
=&(\left | A \right |+\left | C \right |+1)^{k}\left [ \frac{\left | A \right |\left | q \right |^{k}}{k+1}+1+\frac{\left | C \right | }{k+1}\right ]\\\nonumber
<&(\left | A \right |+\left | C \right |+1)^{k}(\left | A \right |+\left | C \right |+1)\\\nonumber
=&(\left | A \right |+\left | C \right |+1)^{k+1}      
\end{align} 
Hence, by mathematical induction, we obtain that \begin{align}\label{E22}
 \left | a_{n} \right |\leq \left (\left | A \right |+\left | C \right |+1\right )^{n} 
\end{align}
for all $n\in \mathbb{N} \cup\left \{ 0 \right \}.$ Then the radius $$R=\frac{1}{\limsup_{n \to \infty} \sqrt[n]{\left | a_{n} \right | } } \geq\frac{1}{\left | A \right |+\left | C \right |+1  }>0 $$ of convergence of $g(z)$ 
Combining Proposition~\ref{P2}, it follows that $g(z)$ is an entire function in the complex plane. Note that $g(z)$ is not a polynomial, then $g(z)$ is a transcendental entire solution of \eqref{E19}. 

$\mathbf{Step}$ $\mathbf{2.3}:$ 
If $\left | a_{0} \right | $=$\left |\beta   \right |>1.$ Suppose that \begin{align}\label{E45}
\left | a_{n} \right |\leq \left |\beta  \right |^{n+1}\left ( \left | A \right |+\left | C \right |+\left | \beta \right | \right )^{n}, \,\,\,\forall n\leq k.
\end{align}
Then the assumption holds whenever $n=0.$ We next prove that \eqref{E45} also holds for $n=k+1.$
Combining \eqref{E51} with \eqref{E45}, it follows that 
\begin{align}\label{E48}
&\left | a_{k+1} \right |\\\nonumber
\leq &\frac{1}{k+1} \left | \beta  \right |^{k+1}(\left | A \right |+\left | C \right |+\left | \beta \right |  )^{k}\left [ \left | A \right |+(k+1)\left | \beta  \right |+\left | C \right |   \right ]\\\nonumber
<&(\left | A \right |+\left | C \right |+\left | \beta \right |  )^{k}\left | \beta  \right |^{k+1}(\left | A \right |+\left | \beta \right |  +\left | C \right |)\\\nonumber
\leq &(\left | A \right |+\left | C \right |+\left | \beta \right |  )^{k+1}\left | \beta  \right |^{k+2}      
\end{align} 
Hence, by mathematical induction, we obtain that \begin{align*}
 \left | a_{n} \right |\leq \left |\beta  \right |^{n+1}\left ( \left | A \right |+\left | C \right |+\left | \beta \right | \right )^{n}, 
\end{align*}
for all $n\in \mathbb{N} \cup\left \{ 0 \right \}.$ By the same way as Step $2.2,$ we obtain that $g(z)$ is a transcendental entire solution of \eqref{E19}.

(ii). $\mathbf{Case}$ $\mathbf{2}:$ Suppose now that $D\ne 0.$

$\mathbf{Case}$ $\mathbf{2.1}:$ Let $a_{0}=0$ in \eqref{E2}, which is equivalent to $f(0)=0.$ Note that $A,$ $B,$ $C,$ $D$ and $q$ are all fixed constants, then by \eqref{E18}, we see that all coefficients $a_{n}$ of \eqref{E2} are fixed constants, and $a_{1}=BD\ne 0.$ 
In other words, \eqref{E2} is a fixed power series and
\begin{align}\label{E50}
g(z)=BDz+\sum_{n=2}^{\infty}a_{n}z^{n}.    
\end{align}
There are two cases for the value of $\left | BD \right |.$
That is, $\left | BD \right |>1 $ and $0<\left | BD \right |\leq 1.$

If $0<\left | BD \right |\leq 1,$ then $\left | a_{1} \right |=\left | BD \right |\leq \left ( \left | A \right |+\left | C \right |+1  \right ).$
Suppose now that \begin{align}\label{E49}
\left | a_{n} \right |\leq \left ( \left | A \right |+ \left | C \right |+1  \right )^{n}, \forall n\leq k.
\end{align}
Combining \eqref{E51}, \eqref{E49}, and $0<\left | q \right |<1,$
we have
\begin{align*}
&\left | a_{k+1} \right |\\\nonumber
\leq& \frac{1}{k+1}\left [ \left | A \right |\left | a_{k} \right |+\sum_{i,j=0}^{i+j=k}\left | a_{i}a_{j} \right |+\left | C \right |\left | a_{k} \right | \right ]\\\nonumber
 \leq&\frac{1}{k+1}\left [ \left | A \right |\left (\left | A \right |+\left | C\right | +1 \right )^{k}+(k+1)\left (\left | A \right |+\left | C\right | +1 \right )^{k}+\left | C \right | \left (\left | A \right |+ \left |C\right | +1 \right )^{k}\right ] \\\nonumber  
 \leq& \left (\left | A \right |+\left | C\right | +1 \right )^{k+1}.
\end{align*}

If $\left | BD \right |>1,$ then $\left | a_{1} \right |=\left | BD \right |\leq \left | BD \right |(\left | A \right |+\left | C \right |+1 ).$ Suppose that \begin{align}\label{E90}
 \left | a_{n} \right |\leq \left | BD \right |^{n}\left ( \left | A \right |+ \left | C \right |+1 \right )^{n}, \forall n\leq k.   
\end{align}
Together,  \eqref{E51}, \eqref{E90}, and $0<\left | q \right |<1 $ yield
\begin{align*}
&\left | a_{k+1} \right |\leq \frac{1}{k+1}\left [ \left | A \right |\left | a_{k} \right |+\sum_{i,j= 0}^{i+j = k}\left | a_{i}a_{j} \right |+\left | C \right |\left | a_{k} \right | \right ]\\\nonumber
 \leq &\left | BD \right |^{k}(\left | A \right |+\left | C \right |+1 )^{k}\left [ \frac{\left | A \right | }{k+1}+1+\frac{\left | C \right | }{k+1} \right ] \\\nonumber  
\leq &\left | BD \right |^{k} (\left | A \right |+\left | C \right |+1 )^{k+1}  \\\nonumber
\leq &\left | BD \right |^{k+1} (\left | A \right |+\left | C \right |+1 )^{k+1}  
\end{align*}
By the proof of Step $2.2,$ it follows that \eqref{E50} is an entire solution of \eqref{E19} regardless of whether $\left | BD \right |>1$ or $0<\left | BD \right |\leq1.$ We now prove that \eqref{E50} must be a transcendental entire solution of \eqref{E19}. Assume, on the contrary, that \eqref{E50} is a polynomial in $z$ with $\deg_z g(z)=G.$ 
From \eqref{E50}, we see that  $\deg_z g(z)=G\geq 1.$  
By comparing the degrees of the two sides of \eqref{E19}, we obtain a contradiction. The assertion follows.

$\mathbf{Case}$ $\mathbf{2.2}:$ Let $a_{0}\ne 0$ in \eqref{E2}, which is equivalent to $f(0)\ne 0.$ Since 
$A,$ $B,$ $C,$ $D,$ and $q$ are fixed constants, it follows from
\eqref{E18} that all coefficients $a_{n}$ of \eqref{E2} are polynomials in $a_{0}.$ Note that $a_{0}\in \mathbb{C}\setminus \left \{ 0 \right \}$ is not fixed. Referring to Step $1$ of Case $1,$ we can first prove that there exist uncountably many $a_{0}$ such that each such $a_{0}$ ensure that \eqref{E2} has infinitely many nonzero coefficients $a_{n}.$ Then by fixing this value of $a_{0},$ we use mathematical induction to prove that $g(z)$ is an entire solution.
This process is  exactly the same case as in Case $1,$ we will omit its proof here. In this case, we can obtain $g(z)$ is a transcendental entire solution.

(iii) Suppose that the exponential polynomial \eqref{E37} is an entire solution of \eqref{E1}.

According to the form of \eqref{E37}, we can easily obtain the following result:\begin{align}\label{E38}
 f(z)^{2} & = \sum_{i,j= 0}^{m}H_{i}(z)H_{j}(z)e^{(\omega_{i}+\omega_{j})z^{t}},\\\nonumber
f(qz) &=\sum_{l=0}^{m}H_{l}(qz)e^{k_{l}z^{t}},\\\nonumber
{f}'(z) &=\sum_{s=0}^{m}\eta_{s}(z)e^{\omega_{s}z^{t}},\end{align}
where $\omega_{0}=0,k_{l}=\omega_{l}q^{t},$ $\eta_{0}(z)=H'_{0}(z),$ $\eta_{s}(z)=H'_{s}(z)+H_{s}(z)\omega_{s}tz^{t-1}$ and $\omega_{i}$ $(1\leq i\leq m)\in \mathbb{C}.$ Substituting \eqref{E37} and \eqref{E38} into \eqref{E1}, we have 
\begin{align}\label{E39}
&\sum_{s= 0}^{m}\eta_{s}(z)e^{\omega_{s}z^{t}}\\\nonumber
=& A\sum_{l = 0}^{m}H_{l}(qz)e^{k_{l}z^{t}}+B\sum_{i,j= 0}^{m}H_{i}(z)H_{j}(z)e^{(\omega_{i}+\omega_{j})z^{t}}+C\sum_{h= 0}^{m}H_{h}(z)e^{\omega_{h}z^{t}}+D.          
\end{align}
Note that $\omega_{1},\dots,\omega_{m}$ are pairwise different constants and $0<\left | q \right |<1 $. Let $\left | \omega_{1} \right |<\left |\omega_{2} \right |<\dots\left |\omega_{m} \right |.$
By applying Corollary~\ref{C1} to \eqref{E39}, it follows that $H_{m}(z)\equiv 0.$ This contradicts the fact that $H_{j}(z)\not\equiv 0$ for all $1\leq j\leq m.$ Therefore, the assertion follows.

(iv) (a) Let $z=z_{0}$ be  any pole of $f(z),$ and let
$f(z_{0})=\infty^{t}$ with $t\ge 2.$ If $z_{0}=0,$ by \eqref{E1}, it follows that 
$${f}'(0)=Af(0)+Bf(0)^{2}+Cf(0)+D.$$ This is a contradiction with $f(0)=\infty^{t},t\ge 2.$
Therefore, $z_{0}\ne 0.$
From \eqref{E1}, $f(z)$ must have a pole of order $2t$ at $z=qz_{0}.$
Continuing this process, we have $f(z)$ has a pole of order $4t$ at $z=q^{2}z_{0},$ and inductively, we have $f(z)$ has a pole of order $2^{n}t$ at $z=q^{n}z_{0}.$ 
Note that
$0<\left|q\right|<1,$ then $q^{n}z_{0}\longrightarrow 0 $ as $n\longrightarrow\infty,$ which means the origin is the limit point of poles of $f(z).$
This contradicts the fact that $f(z)$ is a meromorphic function. Therefore, $t=1.$

Next, we prove that \eqref{E19} admits  both a global and a local meromorphic solution. According to Proposition~\ref{P2}, if we can prove that \eqref{E19} has a local meromorphic solution around $z=0,$
then we can extend the solution to the entire complex plane. We will consider the  Laurent series expansion of $g(z)$ in a neighborhood of $z=0$ and $z\ne0.$
To this end, consider a formal power series 
\begin{align}
\label{E3} g(z)=b_{-1}z^{-1}+ \sum_{n=0}^{\infty}b_{n}z^{n}=b_{-1}z^{-1}+l(z), \quad b_{-1}\ne 0,     
\end{align}
where $l(z)= \sum_{n=0}^{\infty}b_{n}z^{n}.$     
Then
\begin{align}\label{E15} 
&{g}'(z)=\frac{-b_{-1}}{z^{2}}+\sum_{n=1}^{\infty}nb_{n}z^{n-1},\\\nonumber
&g(qz)=\frac{b_{-1}}{qz}+\sum_{n=0}^{\infty}b_{n}q^{n}z^{n},\\\nonumber
&g(z)^{2}=\frac{b_{-1}^{2}}{z^{2}}+\sum_{n=0}^{\infty}\left( \sum_{i,j=0}^{i+j=n}b_{i}b_{j}\right)z^{n}+\sum_{n=0}^{\infty}2b_{-1}b_{n}z^{n-1}
 \end{align}
Substituting \eqref{E3}, \eqref{E15} into \eqref{E19}, we have \begin{align}\label{E16}
&-b_{-1}=b_{-1}^{2},\\\nonumber
&Ab_{-1}q^{-1}+2b_{-1}b_{0}+Cb_{-1}=0,\\\nonumber 
&b_{1}=Ab_{0}+b_{0}^{2}+Cb_{0}+2b_{-1}b_{1}+BD,\\\nonumber
&\dots\\\nonumber 
&  (n+1)b_{n+1}=Ab_{n}q^{n}+\sum_{i,j=0}^{i+j=n}b_{i}b_{j}+Cb_{n}+2b_{-1}b_{n+1}\quad n\geq 1.           \end{align}
From \eqref{E16}, a simple calculation  obtains that 
\begin{align}\label{E23}
&b_{-1}=-1,\\\nonumber
&b_{0}=-\frac{1}{2}\left(\frac{A}{q}+C\right),\\\nonumber
&\dots\\\nonumber
&(n+1)b_{n+1}=b_{n}(Aq^{n}+C)+2b_{n+1}b_{-1}+\sum_{i,j=0}^{i+j=n}b_{i}b_{j}.
\end{align}
Since $A,$ $B,$ $C,$ $D$ and $q$ are all fixed constants, we conclude that all coefficients of \eqref{E3} are specific constants, and they depend on $A,$ $B,$  $C,$ $D$ and $q.$
In order to prove that \eqref{E3} is a convergent power series, we only need to prove that $l(z)=\sum_{n=0}^{\infty}b_{n}z^{n}$ is convergent. By \eqref{E23}, we have
\begin{align}\label{E52}
 b_{n+1}=\frac{1}{n+1}\left [ Ab_{n}q^{n}+\sum_{i,j=0}^{i+j=n}b_{i}b_{j}+Cb_{n}+2b_{-1}b_{n+1}\right ],  n\geq 1.    
 \end{align}
 Let $b_{0}\ne 0.$ Then there are two cases in which $0<\left | b_{0} \right |<1 $ and $\left | b_{0} \right |>1 .$  
 This case is similar to Step $2.2$ and Step $2.3.$ That is, we can use mathematical induction to prove that $l(z)=\sum_{n=0}^{\infty}b_{n}z^{n}$ is an entire function. 
Therefore,
we conclude that \eqref{E3} is a specific global meromorphic solution to \eqref{E19}.

If $b_{0}=0,$ which is equivalent to $C=-\frac{A}{q} ,$ then it follows from \eqref{E16} that $b_{1}=\frac{1}{3}BD .$ 
If $D\ne 0,$ then $b_{1}\ne 0.$ That is, $b_{1}$ is the first nonzero term of the coefficients $\left \{ b_{n} \right \}_{n=0}^{\infty}  .$ 
Let $b_{1}=\gamma\ne 0.$  We can prove that $\left | b_{n}  \right |\leq 2^{n}\left | \gamma \right | ^{n+1}$ for all $n\in \mathbb{N}\cup \left \{ 0 \right \}.$ This process is similar to the case where $b_{0}\ne 0$ ($b_{0}$ is the first nonzero term of the coefficients $\left \{ b_{n} \right \}_{n=0}^{\infty}  $ ). Therefore, we will omit its proof details here. In this case, we further prove that \eqref{E3} must be a transcendental meromorphic solution to \eqref{E19}. Suppose, on the contrary, that $g(z)$ is a rational function. From \eqref{E3} and \eqref{E23}, we see that \begin{align}\label{E102} 
g(z)=\frac{-1}{z}+\frac{1}{3}BDz+\cdots   
\end{align}
Assume that $R(z)$ is an irreducible rational function in $z.$ We define the degree of $R(z)$ by $$\deg_zR(z)=\max \left \{ \deg_zm(z),\deg_zn(z)   \right \},\,\,\,R(z)=\frac{m(z)}{n(z)}.$$ 
Let $\deg_zg(z)=N$  $(N\geq 2 ).$
Substituting \eqref{E102} into \eqref{E19}, and then by comparing the degrees of the two sides of \eqref{E19}, we  obtain a contradiction. Therefore, $g(z)$ must be a transcendental meromorphic function.
If $D=0,$ then $b_{0}=b_{1}=0.$ By \eqref{E16}, we see that $b_{n}=0$ for all $n\in \mathbb{N}\cup \left \{ 0 \right \}.$ Then by \eqref{E3} and Proposition \ref{P2},
we have $g(z)=-\frac{1}{z}$ is a global meromorphic solution to \eqref{E19}.

(b) We now consider the following formal power series:
\begin{align}
\label{E34} g(z)=b_{-1}(z-z_{0})^{-1}+ \sum_{n=0}^{\infty}b_{n}(z-z_{0})^{n},\quad z_{0}\ne 0. 
\end{align}
Then \begin{align}\label{E35}  
&{g}'(z)=-b_{-1}(z-z_{0})^{-2}+\sum_{n=1}^{\infty}nb_{n}(z-z_{0})^{n-1},\\\nonumber
&g(z)^{2}=b_{-1}^{2}(z-z_{0})^{-2}+\sum_{n=0}^{\infty}2b_{-1}b_{n}(z-z_{0})^{n-1}+\sum_{n=0}^{\infty}\left (\sum_{i,j=0}^{i+j=n}b_{i}b_{j}\right )(z-z_{0})^{n}, \\\nonumber  
&g(qz)=b_{-1}(qz-z_{0})^{-1}+\sum_{n=0}^{\infty}b_{n}(qz-z_{0})^{n}       
\end{align}
Substituting \eqref{E34} and  \eqref{E35} into \eqref{E19}, and using the binomial theorem $$(r+s)^{n}=\sum_{k=0}^{n}C_{n}^{k}r^{n-k}s^{k}, \quad C_{n}^{k}=\frac{n!}{k!(n-k)!},$$
we have 
\begin{align}\label{E41} 
b_{-1}= -1, (n+3)b_{n+1}= Ab_{n}q^{n}+\sum_{i,j= 0}^{i+j = n}b_{i}b_{j}+Cb_{n},\quad n\geq 1. 
\end{align}
We can see that $b_{0}$ depends on $z_{0}.$
In order to discuss the value of $b_{0},$ we separate three cases:

$\mathbf{Case}$ 
$\mathbf{A}:$
Assume that $b_{n}\equiv 0$ for all $n\geq 0.$ Then, it follows from \eqref{E34} that
$g(z)=\frac{-1}{z-z_{0}}$ is a local meromorphic solution 
to \eqref{E19} in a neighborhood of $z=z_{0}.$

$\mathbf{Case}$  $\mathbf{B}:$
If $b_{0}\ne 0,$ then we can refer to Step $2.2$ and Step $2.3$ to prove that $$R=\frac{1}{\limsup_{n \to \infty}\sqrt[n]{\left |b_{n}\right|}} >0.$$
That is $\sum_{n=0}^{\infty}b_{n}(z-z_{0})^{n},$ $z_{0}\ne 0,$ is a locally analytic function. In other words, \eqref{E34} is a locally meromorphic function of \eqref{E19}.

$\mathbf{Case}$  $\mathbf{C}:$
Suppose there exists a $m$ ($m\geq 1$) such that $b_{m}\ne 0$ and all $b_{n}\equiv 0,$ $n\leq m-1.$ Then \eqref{E34} reduces to $$g(z)=\frac{-1}{z-z_{0}}+\sum_{n=m}^{\infty}b_{n}(z-z_{0})^{n},\quad z_{0}\ne 0.$$  
Let $b_{m}=\gamma.$ We discuss the two cases in which $\left | \gamma  \right |$ satisfies $0<\left | \gamma \right |<1$ and $\left | \gamma \right |>1,$ respectively.

Let $\left | \gamma  \right |> 1.$ Then $ \left | b_{m} \right |=\left | \gamma \right |  < \left | \gamma  \right |^{m+1}\left ( \left | A \right |+\left | C \right |+\left | \gamma  \right |    \right )^{m} .$ Suppose that \begin{align}\label{E53}
\left | b_{n} \right |< \left | \gamma  \right |^{n+1}(\left | A \right |+\left | C \right |+\left | \gamma \right |)^{n}, \quad \forall n\leq k, \,\,\, k> m.  
\end{align}
Combining \eqref{E41}, \eqref{E53}, and $0<\left | q \right |<1.$ Then we see that
\begin{align}
\left | b_{k+1} \right |
&=\frac{1}{k+3}\left [ Ab_{k}q^{k}+\sum_{i,j=0}^{i+j=k}b_{i}b_{j}+Cb_{k}  \right ]  \\\nonumber
&\leq \frac{1}{k+3}\left [ \left | A \right |\left | b_{k} \right |+\sum_{i,j=0}^{i+j=k}\left | b_{i}b_{j} \right |+\left | C \right |\left | b_{k} \right | \right ]                                               \\\nonumber
&\leq \frac{1}{k+3}\left | \gamma  \right |^{k+1}(\left | A \right |+\left | C \right |+\left | \gamma  \right |  )^{k}\left [ \left | A \right |+(k+1)\left | \gamma  \right |+\left | C \right | \right ]\\\nonumber    
&\leq \left | \gamma  \right |^{k+1}(\left | A \right |+\left | C \right |+\left | \gamma\right | )^{k}\left [ \left | A \right |+\left | C \right |+\left | \gamma  \right |  \right ]\\\nonumber 
&< \left | \gamma  \right |^{k+2}(\left | A \right |+\left | C \right |+\left | \gamma\right | )^{k+1}.
\end{align} 
Therefore, by mathematical induction, we have \begin{align*} \left | b_{n} \right |< \left | \gamma  \right |^{n+1}(\left | A \right |+\left | C \right |+\left | \gamma \right |)^{n}, \quad \forall n\in \mathbb{N}\cup \left \{ 0 \right \}.  
\end{align*}

Let $0<\left | \gamma  \right |\leq 1.$ 
Then $\left | \gamma  \right | =\left | b_{m} \right |\leq (\left | A \right |+\left | C \right |+1 )^{m}.$ Suppose that \begin{align}\label{E54}
\left | b_{n} \right |\leq (\left | A \right |+\left | C \right |+1 )^{n},\quad \forall n\leq k, \,\,\,k>m.
\end{align}
Together, \eqref{E41}, \eqref{E54}, and $0<\left | q \right |<1 $ yield \begin{align*}
b_{k+1}&=\frac{1}{k+3}\left [ Ab_{k}q^{k}+\sum_{i,j=0}^{i+j=k} b_{i}b_{j} + Cb_{k} \right ]\\
 &\leq \frac{1}{k+3}\left [ \left | Ab_{k} \right |+\sum_{i,j=0}^{i+j=k}\left | b_{i}b_{j} \right |+\left | cb_{k} \right |     \right ] \\
&\leq \left ( \left | A \right |+\left | C \right |+1   \right )^{k}\left [ \frac{\left | A \right | }{k+3}+\frac{k+1}{k+3}+\frac{\left | C \right | }{k+3}  \right ] \\
&\leq \left ( \left | A \right |+\left | C \right |+1   \right )^{k+1}.
\end{align*}
By mathematical induction, we conclude that 
\begin{align*}
\left | b_{n} \right |\leq  (\left | A \right |+\left | C \right |+1 )^{n},\quad \forall n\in \mathbb{N}\cup\left \{ 0 \right \}.
\end{align*}
Referring to Steps $2.2$ and $2.3$, we can always obtain that the radius of convergence $R>0$ of $\sum_{n=m}^{\infty}b_{n}(z-z_{0})^{n}$ ($z_{0}\ne 0$). That is, $\sum_{n=m}^{\infty}b_{n}(z-z_{0})^{n}$ is  a locally analytic function. In other words, $$g(z)=\frac{-1}{z-z_{0}}+\sum_{n=m}^{\infty}b_{n}(z-z_{0})^{n},\quad z_{0}\ne 0,$$ is a local meromorphic solution of \eqref{E19}, regardless of whether $0<\left | \gamma  \right |\leq 1 $ or $\left | \gamma  \right |> 1.$

(v) Let $f(z)$ be a meromorphic solution of \eqref{E1}, and let $f(z_{0})=\infty^{t}$ with $t\geq 2.$
By the beginning of (iv), we see that $f\left ( q^{n}z_{0} \right )=\infty^{2^{n}t}$ $(z_{0}\ne 0).$ Let $r=\left | q^{n}z_{0} \right | .$ Then, by the definition of exponent of convergence  for the poles of $f(z),$ we see that 
$$\lambda \left ( \frac{1}{f}  \right )=\limsup_{r \to \infty}\frac{\log n(r,f)}{\log r}\geq \limsup_{n \to \infty}\frac{\log (2^{n}t)}{\log \left | q^{n}z_{0} \right | }=\frac{\log 2}{\log q}.$$
This leads to a contradiction with the fact that the order of $f(z)$ is zero. That is, all poles of $f(z)$ are simple. Since $f(z)=\frac{g(z)}{B} $, all poles of $g(z)$ are also simple.
Let \begin{align}\label{E43}  
g(z)=\frac{c_{-1}}{z-z_{0}}+\sum_{n=0}^{\infty}c_{n}(z-z_{0})^{n}=\frac{c_{-1}}{z-z_{0}}+m(z), \quad z_{0}\in \mathbb{C},  
\end{align}
where $\sum_{n=0}^{\infty}c_{n}(z-z_{0})^{n}=m(z),$
$c_{-1}\ne 0.$
Then 
\begin{align}\label{E44}  
&{g}'(z)=-c_{-1}(z-z_{0})^{-2}+\sum_{n=1}^{\infty}nc_{n}(z-z_{0})^{n-1},\\\nonumber
&g(z)^{2}=c_{-1}^{2}(z-z_{0})^{-2}+\sum_{n=0}^{\infty}2c_{-1}c_{n}(z-z_{0})^{n-1}+\sum_{n=0}^{\infty}\left (\sum_{i,j=0}^{i+j=n}c_{i}c_{j}\right )(z-z_{0})^{n}, \\\nonumber  
&g(qz)=c_{-1}(qz-z_{0})^{-1}+\sum_{n=0}^{\infty}c_{n}(qz-z_{0})^{n}.       
\end{align}
Substituting \eqref{E43} and \eqref{E44} into \eqref{E19}, by the similar way as \eqref{E35}, we have
\begin{align}\label{E46}
&c_{-1}=-1,\\\nonumber
&c_{n+1}= \frac{1}{n+3}\left [Ac_{n}q^{n}+\sum_{i,j = 0}^{i+j= n}c_{i}c_{j}+Cc_{n} \right ],\quad n\geq 1,  
\end{align}
where $c_{0}$ depends on $z_{0}.$ Note that $c_0$ and $c_1$ are both real numbers. From \eqref{E46}, we have $c_2,c_3,\cdots$ are all real numbers.
Suppose there exists a point $z_{0}\in \mathbb{R}$ such that $c_{0}=m(z_{0})=\sum_{n=0}^{\infty}c_{n}(z_{0}-z_{0})^{n} >1.$ Then $ c_0 \geq  c_0.$ Suppose now that $c_n \geq c_0   $ for all $n\leq k.$ 
By \eqref{E46}, we have
\begin{align}\label{E47}
c_{k+1}&= \frac{1}{k+3}\left [Ac_{k}q^{k}+\sum_{i,j = 0}^{i+j= k}c_{i}c_{j}+Cc_{k} \right ] \\\nonumber 
       &\geq \frac{1}{k+3}\left [Ac_{0}q^{k}+c_{0}^{2}(k+1)+Cc_{0}  \right ]\\\nonumber 
        &=c_{0}\left [ \frac{A}{k+1}q^{k}+c_{0}+\frac{C}{k+1}\right]\\\nonumber 
&>c_{0}^{2}\\\nonumber
        &>c_{0}.
\end{align}
Hence, by mathematical induction, we have $c_{n}\geq c_{0}>1$ for all $n\in \mathbb{N}\cup \left \{ 0 \right \}.$ This means $\lim_{n \to \infty} c_{n} \geq c_{0}> 1>0,$ it follows that $m(z)=\sum_{n=0}^{\infty}c_{n}(z-z_{0})^{n}$ is divergent in $z=z_{0}+1.$ By Abel’s theorem, $m(z)$ diverges for all $\left | z-z_0 \right |>1 .$
That is, \eqref{E43} cannot be a global meromorphic solution to \eqref{E19}, which contradicts with our assumption. Then we have  $m(z_{0})=\sum_{n=0}^{\infty}c_{n}(z_{0}-z_{0})=c_{0} \leq 1$ for all $z_{0}\in \mathbb{C}.$ Combining Lemma~\ref{L3}, it follows that $m(z)$ must be a constant. That is, $g(z)$ is a rational function.

\begin{remark} In Case $A,$ $B,$ and $C$ of (iv), we can prove that the radius of convergence $R>0$ of $g(z).$ However, it is possible that $R=+\infty.$ In this case, a local meromorphic solution in a neighborhood of $z=z_{0}$ to \eqref{E19} can be replaced by a global meromorphic solution to \eqref{E19}. 
\end{remark}
\end{proof}

\section{proof of theorem \ref{T2}}
\begin{proof}
(i) Under the assumption, \eqref{E1} reduces to 
\begin{align}\label{E13}
{f}'(z)=Af(qz)+Cf(z)+D. 
\end{align}
Consider a formal series
\begin{align}\label{E12}
f(z)=\sum_{n=0}^{\infty}a_{n}z^{n}. 
\end{align}
Substituting \eqref{E12} into \eqref{E13}, it follows that

\begin{align}\label{E26}
&a_{1}=Aa_{0}+Ca_{0}+D,\\\nonumber
&2a_{2}=Aa_{1}q+Ca_{1},\\\nonumber
&3a_{3}=Aa_{2}q^{2}+Ca_{2},\\\nonumber
&\cdots \\\nonumber
&(n+1)a_{n+1}=Aa_{n}q^{n}+Ca_{n},\,\,\,n\geq 1.       
 \end{align}
Therefore, we  have the following:
\begin{align}\label{E14}
 \frac{1}{R} =\lim_{n \to \infty}\left|\frac{a_{n+1}}{a_{n}}\right|=\lim_{n \to \infty}\left|\frac{Aq^{n}+C}{n+1}\right|,     
\end{align}  
where $R$ denotes the radius of convergence of $f(z).$ It follows that if $0<\left | q \right |<1,$ then $R=\infty;$ that is to say, \eqref{E12} is an entire  solution of \eqref{E13}. 
From \eqref{E26} we see that $a_n$ forms a one parameter family, determined by  $a_0,$ for all $n\in \mathbb{N}.$
From Step $1$ of the proof of Theorem \ref{T1}, we obtain 
there exists uncountably many $a_0\in \mathbb{C}$ such that all the coefficients of \eqref{E12} are nonzero. Consequently, \eqref{E13} has uncountably many transcendental entire solutions. However, if $\left | q \right |>1, $ $R=0;$ that is to say, \eqref{E12} converges only at the origin.

(ii) We first consider the case where $0<\left | q \right |<1.$
Suppose, on the contrary, that $f(z)$ is a meromorphic solution to \eqref{E13}.
Let $f(z_{0})=\infty^{t}.$ If $z_{0}=0,$ then 
\begin{align}\label{E27} {f}'(0) & = Af(0)+Cf(0)+D. 
\end{align}
By comparing the orders of both sides of \eqref{E27}, we can obtain a contradiction. Hence, $z_{0}\ne 0.$ Then $f(qz_{0})=\infty ^{t+1}.$ Replacing $z$ with $qz$ in \eqref{E13}, it follows that \begin{align}\label{E28}
{f}'(qz) = Af(q^{2}z)+Cf(qz)+D. 
\end{align}
Then $f(q^{2}z_{0})=\infty^{t+2}.$ By repeating these steps, it follows that $f(q^{n}z_{0})=\infty^{t+n}.$ Note that $0<\left|q\right|<1.$ Then $q^{n}z_{0}\longrightarrow 0 $ as $n\longrightarrow\infty,$ which means the origin is the limit point of poles of $f(z).$
This contradicts the fact that $f(z)$ is a meromorphic function. 
Therefore, our assumption does not holds. That is, \eqref{E13} has no meromorphic solutions when $0<\left | q \right |<1.$
\end{proof}    
\section*{Declarations}

\subsection*{Conflict of interest}
The authors have no conflicts of interest to declare that are relevant to the content of this article.

\bibliographystyle{siam}
\bibliography{name}

@article{gundersen2002meromorphic,
  author = {Gundersen, G.G. and Heittokangas, J. and Laine, I. and Rieppo, J. and Yang, D.},
  title = {Meromorphic solutions of generalized Schr{\"o}der equations},
  journal = {Aequationes Mathematicae},
  volume = {63},
  pages = {110--135},
  year = {2002},
  doi = {10.1007/s00010-002-8010-z}
}

@book{laine2011nevanlinna,
  title={Nevanlinna theory and complex differential equations},
  author={Laine, Ilpo.},
  volume={15},
  year={2011},
  publisher={Walter de Gruyter}
}

@article{bergweiler1998meromorphic,
  title={Meromorphic solutions of some functional equations},
  author={Bergweiler, Walter and Ishizaki, Katsuya and Yanagihara, Niro.},
  journal={Methods and Applications of Analysis},
  volume={5},
  pages={248--258},
  year={1998},
  publisher={INTERNATIONAL PRESS}
}

@article{10.1112/jlms/jdm073,
    author = {Laine, Ilpo and Yang, Chung.C.},
    title = {Clunie theorems for difference and q-difference polynomials},
    journal = {Journal of the London Mathematical Society},
    volume = {76},
    number = {3},
    pages = {556-566},
    year = {2007},
    month = {10},
    issn = {0024-6107},
    doi = {10.1112/jlms/jdm073},
    url = {https://doi.org/10.1112/jlms/jdm073},
    eprint = {https://academic.oup.com/jlms/article-pdf/76/3/556/2415708/jdm073.pdf},
}

@article{f88b8918-bbff-38b1-8dbf-6a860ea14875,
 ISSN = {00029947},
 URL = {http://www.jstor.org/stable/1999601},
 author = {Lee A. Rubel},
 journal = {Transactions of the American Mathematical Society},
 number = {1},
 pages = {43--52},
 publisher = {American Mathematical Society},
 title = {Some Research Problems about Algebraic Differential Equations},
 urldate = {2025-11-04},
 volume = {280},
 year = {1983}
}

@article{Ritt1926,
author = {Ritt, J.F.},
journal = {Mathematische Annalen},
pages = {671-682},
title = {Transcendental transcendency of certain functions of Poincaré},
url = {http://eudml.org/doc/159150},
volume = {95},
year = {1926},
}

@article{wen2012exponential,
  title={Exponential polynomials as solutions of certain nonlinear difference equations},
  author={Wen, Zhi.T and Heittokangas, Janne and Laine, Ilpo.},
  journal={Acta Mathematica Sinica, English Series},
  volume={28},
  pages={1295--1306},
  year={2012},
  publisher={Springer}
}

@article{zhang2010nevanlinna,
  title={On the Nevanlinna characteristic of f (qz) and its applications},
  author={Zhang, Ji.L and Korhonen, Risto},
  journal={Journal of Mathematical Analysis and Applications},
  volume={369},
  number={2},
  pages={537--544},
  year={2010},
  publisher={Elsevier}
}

@book{liu2021complex,
  title={Complex delay-differential equations},
  author={Liu, Kai and Laine, Ilpo and Yang, Lian.Z.},
  volume={78},
  year={2021},
  publisher={Walter de Gruyter GmbH \& Co KG}
}

@article{li2016existence,
  title={On existence of solutions of differential-difference equations},
  author={Li, Hai.C.},
  journal={Mathematical Methods in the Applied Sciences},
  volume={39},
  number={1},
  pages={144--151},
  year={2016},
  publisher={Wiley Online Library}
}

@article{zhao2023meromorphic,
  title={On Meromorphic Solutions of Non-linear Differential-Difference Equations},
  author={Zhao, Ming.X and Huang, Zhi.G.},
  journal={Journal of Nonlinear Mathematical Physics},
  volume={30},
  number={4},
  pages={1444--1466},
  year={2023},
  publisher={Springer}
}

@article{zhang2011entire,
  title={On entire solutions of a certain type of nonlinear differential and difference equations},
  author={Zhang, Jie and Liao, Liang.W.},
  journal={Taiwanese Journal of Mathematics},
  volume={15},
  number={5},
  pages={2145--2157},
  year={2011},
  publisher={Mathematical Society of the Republic of China}
}

@article{xu2021meromorphic,
  title={Meromorphic solutions of delay differential equations related to logistic type and generalizations},
  author={Xu, Ling and Cao, Ting.B.},
  journal={Bulletin des Sciences Math{\'e}matiques},
  volume={172},
  pages={103040},
  year={2021},
  publisher={Elsevier}
}

@article{halburd2017growth,
  title={Growth of meromorphic solutions of delay differential equations},
  author={Halburd, Rod and Korhonen, Risto},
  journal={Proceedings of the American Mathematical Society},
  volume={145},
  number={6},
  pages={2513--2526},
  year={2017}
}

@article{zhang2020entire,
  title     = {Entire Solutions of Delay Differential Equations of Malmquist Type},
  author    = {Ran.R. Zhang and Zhi.B. Huang},
  journal   = {Journal of Applied Analysis and Computation},
  volume    = {10},
  number    = {5},
  pages     = {1720--1740},
  year      = {2020},
  issn      = {2156-907X},
  doi       = {10.11948/20190176},
  url       = {http://www.jaac-online.com//article/id/f093aa56-ba64-4ebc-a939-7aebc8f2f363}
}

@article{wang2019quantitative,
  title={Quantitative properties of meromorphic solutions to some differential-difference equations},
  author={Wang, Qiong and Han, Qi and Hu, Pei.C.},
  journal={Bulletin of the Australian Mathematical Society},
  volume={99},
  number={2},
  pages={250--261},
  year={2019},
  publisher={Cambridge University Press}
}

@article{laine2022remarks,
  title={Remarks on meromorphic solutions of some delay-differential equations},
  author={Laine, I and Latreuch, Z.},
  journal={Analysis Mathematica},
  volume={48},
  number={4},
  pages={1081--1104},
  year={2022},
  publisher={Springer}
}

@article{liu2017meromorphic,
  title={Meromorphic solutions of complex differential--difference equations},
  author={Liu, Kai and Song, Chang.J.},
  journal={Results in Mathematics},
  volume={72},
  number={4},
  pages={1759--1771},
  year={2017},
  publisher={Springer}
}

@article{hu2021malmquist,
  title={A Malmquist type theorem for a class of delay differential equations},
  author={Pei.C. Hu and Man.L. Liu},
  journal={Bulletin of the Malaysian Mathematical Sciences Society},
  volume={44},
  number={1},
  pages={131--145},
  year={2021},
  publisher={Springer}
}

@article{chen2022meromorphic,
  title={Meromorphic solutions of a first order differential equations with delays},
  author={Chen, Yu and T.B. Cao},
  journal={Comptes Rendus. Math{\'e}matique},
  volume={360},
  number={G6},
  pages={665--678},
  year={2022}
}

@article{cao2023meromorphic,
  title={Meromorphic solutions of higher order delay differential equations},
  author={T.B. Cao and Y. Chen and R. Korhonen},
  journal={Bulletin des Sciences Math{\'e}matiques},
  volume={182},
  pages={103227},
  year={2023},
  publisher={Elsevier}
}

@article{cao2025transcendental,
  title={Transcendental meromorphic solutions and the complex Schr$\backslash$"$\{$o$\}$ dinger equation with delay},
  author = {T.B. Cao and  R. Korhonen and Wen.L. Liu},
  journal={arXiv preprint arXiv:2505.15310},
  year={2025}
}

@book{cherry2001nevanlinna,
  title={Nevanlinna’s theory of value distribution: The second main theorem and its error terms},
  author={Cherry, William and Ye, Zhuan},
  year={2001},
  publisher={Springer Science \& Business Media}
}

@book{hayman1964meromorphic,
  title={Meromorphic functions},
  author={Hayman, Walter.Kurt},
  volume={78},
  year={1964},
  publisher={Oxford Clarendon Press}
}

@article{Barnett2007NevanlinnaTF,
  title={Nevanlinna theory for the \$q\$-difference operator and meromorphic solutions of \$q\$-difference equations},
  author={David.C. Barnett and Rodney.G. Halburd and W. Morgan and Risto Korhonen},
  journal={Proceedings of the Royal Society of Edinburgh: Section A Mathematics},
  year={2007},
  volume={137},
  pages={457 - 474},
  url={https://api.semanticscholar.org/CorpusID:15470701}
}

@article{ChiangFeng2008,
  author    = {Yik.Man Chiang and Shao.Ji Feng},
  title     = {On the Nevanlinna characteristic of $f(z+\eta)$ and difference equations in the complex plane},
  journal   = {Ramanujan Journal},
  year      = {2008},
  volume    = {16},
  number    = {1},
  pages     = {105--129},
  doi       = {10.1007/s11139-007-9101-1},
  url       = {https://doi.org/10.1007/s11139-007-9101-1},
}

@article{HALBURD2006477,
title = {Difference analogue of the Lemma on the Logarithmic Derivative with applications to difference equations},
journal = {Journal of Mathematical Analysis and Applications},
volume = {314},
number = {2},
pages = {477-487},
year = {2006},
issn = {0022-247X},
doi = {https://doi.org/10.1016/j.jmaa.2005.04.010},
url = {https://www.sciencedirect.com/science/article/pii/S0022247X05003161},
author = {R.G. Halburd and R. Korhonen}
}

@book{valiron1952fonctions,
  author    = {Georges Valiron},
  title     = {Fonctions analytiques},
  year      = {1952},
  publisher = {Presses Universitaires de France},
  address   = {Paris},
  language  = {French}
}

@article{Ishizaki1998,
  author       = {Ishizaki, K.},
  title        = {Hypertranscendency of meromorphic solutions of a linear functional equations},
  journal      = {aequationes mathematicae},
  year         = {1998},
  volume       = {56},
  number       = {3},
  pages        = {271--283},
  doi          = {10.1007/s000100050062},
  url          = {https://doi.org/10.1007/s000100050062},
  issn         = {1420-8903}
}

@article{Wittich1949,
  author       = {Wittich, Hans},
  title        = {Bemerkung zu einer Funktionalgleichung von H. Poincare},
  journal      = {Archiv der Mathematik},
  year         = {1949},
  volume       = {2},
  number       = {2},
  pages        = {90--95},
  doi          = {10.1007/BF02038566},
  url          = {https://doi.org/10.1007/BF02038566},
  issn         = {1420-8938}
}

@article{cf45bf56-3063-3fee-a55f-03d225af2c88,
 ISSN = {03545180, 24060933},
 URL = {https://www.jstor.org/stable/27381760},
 author = {Deniz Agirsevena},
 journal = {Filomat},
 number = {3},
 pages = {759--766},
 publisher = {University of Nis, Faculty of Sciences and Mathematics},
 title = {On the Stability of the Schrödinger Equation with Time Delay},
 urldate = {2025-07-12},
 volume = {32},
 year = {2018}
}

\end{document}